\documentclass[12pt, reqno]{amsart}
\usepackage{amssymb,latexsym,amsmath,amscd,amsfonts}
\usepackage{latexsym}
\usepackage[mathscr]{eucal}
\usepackage{bm}
\usepackage{hyperref}
\usepackage{cleveref}
\usepackage{enumerate}
\usepackage{xcolor}

\voffset = -54pt \hoffset = -54pt \textwidth = 6.3in \textheight = 9.5in \numberwithin{equation}{section}

\newcommand{\beg}{\begin{equation}}
    \newcommand{\eeg}{\end{equation}}
\newcommand{\ben}{\begin{eqnarray*}}
    \newcommand{\een}{\end{eqnarray*}}

\newtheorem{thm}{Theorem}[section]
\newtheorem{cor}[thm]{Corollary}
\newtheorem{lem}[thm]{Lemma}

\newtheorem{prop}[thm]{Proposition}
\numberwithin{equation}{section}
\theoremstyle{definition}
\newtheorem{defn}[thm]{Definition}
\newtheorem{rem}[thm]{Remark}

\newtheorem{eg}[thm]{Example}

\makeatletter \@namedef{subjclassname@2020}{\textup{2020}
Mathematics Subject Classification} \makeatother

\begin{document}

 \title[Spectral radii]{Spectral radii for subsets of Hilbert $C^*$-modules and
 spectral properties of positive maps}
    \author[Bhat, Saha and Sahasrabudhe]{B V Rajarama Bhat, Biswarup Saha and Prajakta Sahasrabuddhe}
    \address[Statistics and Mathematics Unit, Indian Statistical Institute, R V College Post, Bangalore - 560059, India]{}
    \email{bvrajaramabhat@gmail.com}
    \address[]{}
    \email{brsaha27700@gmail.com}
    \address[]{} \email{praju1093@gmail.com}



\begin{abstract}
The notions of  joint and outer spectral radii are extended to the
setting of  Hilbert $C^*$-bimodules.  A Rota-Strang type
characterisation is proved for the joint spectral radius.  In this
general setting, an approximation result for the joint spectral
radius in terms of the outer spectral radius has been established.

This work leads to a new proof of the Wielandt-Friedland's  formula
for the spectral radius of  positive maps. Following an idea of  J.
E. Pascoe,  a positive map called the {\em maximal part\/} has been
associated to any positive map  with non-zero spectral radius, on
finite dimensional $C^*$-algebras.  This provides a constructive
treatment of the Perron-Frobenius theorem. It is seen that the
maximal part of a completely positive map has a very simple
structure and it is irreducible if and only if the original map is
irreducible.

It is observed that algebras generated by tuples of matrices can be
determined and their dimensions can be computed by realizing them as
linear span of Choi-Kraus coefficients of some easily computable
completely positive maps.

    \end{abstract}

    \maketitle

\section{Introduction}\label{Outer:sec:intro}
Throughout this paper $M_m$ denotes the $C^*$-algebra of $m\times m$
complex matrices where $m$ is a natural number and  $\mathcal{A}$
denotes a general unital $C^*$-algebra. It is well-known that the spectral radius of an element $a\in\mathcal{A}$, denoted $r(a)$, can be obtained by the Gelfand formula (for a detailed proof, see Murphy \cite[Theorem 1.2.7]{MurGJ90}):
\[ r(a)=\lim_{n\to \infty} \|a^n\|^{\frac{1}{n}}. \]
Taking a cue from this, the notion of {\em joint spectral radius\/}
for collections of unital normed algebra elements was first introduced by
Rota and Strang \cite{RotGC:StrG60}. For any bounded subset $S$ of
$\mathcal{A}$, they define the joint spectral radius of $S$  by
\begin{equation}\label{outer:eqjoint}
	\rho(S):=\lim\limits_{n\to\infty}\sup\limits_{a_1,\ldots,a_n\in S}\|a_1\cdots a_n\|^{1/n}.
\end{equation}They characterize this quantity as follows. Let
$\mathfrak{N}_{\mathcal{A}}$ be the family of all norms on
$\mathcal{A}$, which are equivalent to the original norm and under
which $\mathcal{A}$ is a normed algebra.  With this notation
(\cite[Proposition 1]{RotGC:StrG60}):
\begin{equation}\label{outer:eqjoint2}
	\rho (S)= \inf\limits_{N\in\mathfrak{N}_{\mathcal{A}}}\,\sup\limits_{a\in S}N(a).
\end{equation}

The problem
of understanding the convergence behaviour of products of matrices
appears in various mathematical projects. It is only natural that the
notion of joint spectral radius plays a crucial role in such
situations. Research advancements in the last few decades have led to
further exploration of the concept of joint spectral radius, and the
literature is vast. We cite investigations by Daubechies and
Lagarias  \cite{DauI:LagJC92}, Berger and Wang  \cite{BerMA:WanY92}
and Elsner \cite{ElsL92}. Computational challenges  in determining the
joint spectral radius spurred the development of efficient
approximation methods. See, for instance,  Parrilo and Jadbabaie
\cite{ParPA:JadA08}, and Blondel and Nesterov \cite{BloVD:NesY05}.
Recently, Pascoe (\cite{PasJE19}, \cite{PasJE21})  examined a special
approximation procedure for the joint spectral radius by introducing
a notion of {\em outer spectral
	radius\/} for $p$-tuple $(A_1, A_2, \ldots , A_p)$ of $m\times m$
matrices via the formula
\begin{equation}\label{pascoes_formula} \widehat{\rho}(A_1,\dots, A_p)=\sqrt{r\left(\sum _{i=1}^p \overline{A_i}
		\otimes A_i\right)}, \end{equation} where $A\otimes B$ denotes the
usual Kronecker product of two matrices and $\overline{A}$ denotes
the complex conjugate of $A$. {Various generalizations of the joint spectral radius have been proposed recently. For instance, Shalit and Shamovich \cite{ShaOM:ShaE25} define a spectral radius for matrices over operator spaces to study simultaneous similarities, whereas our approach is based on completely positive maps and GNS-bimodules.}

Pascoe's notion of the outer spectral radius unveils connections to
completely positive (CP) maps, offering insights into the quantum
dynamics on $M_m$.  The author did notice this connection. However,
all his computations are mostly based on the formula \ref{pascoes_formula}. This amounts to taking a matrix representation of the
completely positive map and using vectorization and other similar
methods that use coordinates. This restricts the applicability of
the theory to completely positive maps on finite dimensions. In this
article, we consider positive maps and completely positive maps on
general $C^*$-algebras. This necessitates a more abstract approach.
In particular, we are led to defining a joint spectral radius for
subsets of Hilbert $C^*$-bimodules and an outer spectral radius for
sequences of elements from Hilbert $C^*$-bimodules.

Here is a brief outline of the contents and the organisation of this
article. In Section 2, we collect some definitions and important
results from the literature essential for our discussion. In Section
\ref{outer:sec:positive}, we begin by proving some basic results
for general positive maps on unital $C^*$-algebras. These are
motivated by similar results obtained by Popescu in \cite{PopG14}.
His focus was $w^*$-continuous positive maps on
$\mathscr{B}(\mathcal{H})$, where $\mathcal{H}$ is a possibly
infinite dimensional Hilbert space and his methods make use of this
setup. We modify the proofs to suit our situation.
Well-known Perron-Frobenius theory tells us that spectral radii of
non-negative matrices are eigenvalues.  In Theorem \ref{outer:thm1}
we show that the spectral radius is a spectral point (but not
necessarily an eigenvalue) for positive maps on general
$C^*$-algebras. This is a known classical result (Karlin (see
\cite[Theorem 4]{KarS59}), Bonsall (see \cite[Lemma 5]{BonFF55}) and
Schaefer (see \cite[Proposition 1]{SchH60}), however, we have it
here by fairly elementary methods.

In Section \ref{outer:sec:HilbertbimoduleSubsets}, we generalize the
notion of the joint spectral radius and the outer spectral radius to
the setup  of Hilbert $C^*$-bimodules. Let $E$ denote any Hilbert
$\mathcal{A}$--$\mathcal{A}$--bimodule, and let $G$ be any bounded
subset of $E$. Taking inspiration from the definition of the joint
spectral radius of a bounded subset of a Banach algebra, we define
the joint spectral radius of $G$, using the interior tensor product $\odot$ on $E$ by
\begin{equation}\label{interior} \rho_{E}(G):=\lim_{n\to \infty}\sup_{\xi_1,\dots ,\xi_n\in
		G}\| \xi_1 \odot \dots \odot \xi_n \|^{\frac{1}{n}}. \end{equation}
We show that this is well-defined and has analogous properties.
Further, to define the outer spectral radius, we first observe that
with any $p$-tuple $A=(A_1,\dots ,A_p)$ of $m\times m$ matrices, one
can associate a completely positive (CP) map $\tau_A:M_m \to M_m$
defined by $\tau_A(X)= \sum_{i=1}^p A_i^*XA_i$. Clearly, $\sum
_i \overline{A}_i \otimes A_i$ is the matrix of $\tau_A^*$ where $\tau_A^*(X)=\sum_{i=1}^p A_iXA_i^*$ and $\widehat{\rho}(A_1,\dots
,A_p )= \sqrt{r(\tau_A^*)}=\sqrt{r(\tau_A)}$, that is, the outer spectral radius
defined in (\ref{pascoes_formula}) is nothing but the square root
of the spectral radius of the associated CP map $\tau_A$. Taking motivation
from this, consider  a finite or countably infinite sequence $G=
(\xi _1, \xi _2, \ldots )$ of elements of $E$ for which
$\tau_G(a):=\sum _i\langle\xi_i,a\xi_i\rangle$ defines a CP map on
$\mathcal{A}$ and define the outer spectral radius of $G$ by
\begin{equation}
	\widehat{\rho}_E(G):=\sqrt{r(\tau_G)}.\end{equation}
Note that the joint
spectral radius is defined for bounded subsets of Hilbert
$C^*$-bimodules, whereas, the outer spectral radius is naturally
defined for (finite or infinite) sequences of elements of Hilbert
$C^*$-bimodules. Here, repeating some elements does make a difference.
However, for any such sequence, the associated joint spectral radius, by definition, is the joint spectral radius of the set $\{\xi_i:
i\geq 1\},$  and here repetitions of terms will not make any
difference.

With definitions of joint spectral radius and outer spectral radius
in place, we can study and compare them. In Theorem
\ref{Outer:thm:Rota}, we give a Rota-Strang type characterisation
for the joint spectral radius. The main difference from classical
Rota-Strang theory is that we are forced to consider equivalent
norms on the full Fock bimodule generated by the Hilbert $C^*$-bimodule
$E$, with some sub-multiplicative properties. Previously, it was enough
to consider norms on the original algebra. This is natural in view
of the definition of the joint spectral radius in (\ref{interior}).

While developing this theory of spectral radii, we observe
connections with some well-known results. A famous theorem of
Wielandt \cite{WieH50} provides a simple way of computing the
spectral radius of an irreducible entry-wise positive matrix
$A=[a_{ij}]_{1\leq i,j\leq m}$ by showing
\[r(A)= \min _{x>0}\max _{1\leq i\leq m}\frac{(Ax)_i}{x_i}.\]
This result was generalized considerably by Friedland \cite{FriS91}.
He showed that  the spectral radius of a positive map $\varphi $
on a unital $C^*$-algebra $\mathcal{A}$ is given by
\begin{align}\label{outer:friedland}
	r(\varphi ) =\inf \{ r(v^{-1}\varphi (v)): v\in \mathcal {A}~~\mbox{is strictly
		positive}\}.
\end{align}
We obtain these results as simple consequences of our analysis of
spectral radii for subsets of Hilbert $C^*$-bimodules in Section \ref{Outer:sec:specrad}.
We obtain inequalities between the joint and outer spectral radii,
culminating in an approximation theorem for the joint spectral radius of
finite sets in terms of outer spectral radius (See Theorem
\ref{outer:thm22}):
\[\rho_E(\xi_1,\ldots,\xi_d)= \lim\limits_{k\to\infty}\widehat{\rho}_{E^{\otimes k}}(\xi_1^{\otimes k},\ldots,\xi_d^{\otimes
	k})^{1/k}.\]
This is a generalization of a similar result for matrices proved by Pascoe \cite{PasJE21}.

In  Section \ref{Outer:sec:positivemap}, we focus on the study of
positive maps on finite dimensional $C^*$-algebras and their
spectral properties. We return to Perron-Frobenius theory.  Perron
(\cite{PerO07-1}, \cite{PerO07-2}) and Frobenius (\cite{FroVG08},
\cite{FroVG12}) discovered that the spectral radius of matrices with
positive entries is an eigenvalue with a corresponding eigenvector
of positive entries. In fact, they established that if the matrix is
irreducible, then its spectral radius is a simple eigenvalue with
unique positive eigenvectors (up to scalar multiplication). Evans
and H\o egh-Krohn \cite{EvaDE:HoeR78} extended this classical result
to finite-dimensional $C^*$-algebras by broadening the notion of
irreducibility to positive maps on such algebras: A positive map
$\varphi$ on $\mathcal{A}$ is said to be irreducible if there is no
non-trivial projection $p\in \mathcal{A}$ for which the hereditary
$C^*$-subalgebra $p\mathcal{A}p$ is invariant under $\varphi$.
Further observations by Albeverio-H\o egh-Krohn \cite{AlbS:HoeR78},
Groh (\cite{GroU81}, \cite{GroU82}), Schrader \cite{SchR01} and
Lagro-Yang-Xiong \cite{LarM:YanWS:XioS17} highlighted the intriguing
spectral properties possessed by strictly positive or irreducible
positive maps.

For a positive map on a $C^*$-algebra, we will call any positive
eigenvector, with spectral radius as the eigenvalue, a
Perron-Frobenius eigenvector. Pascoe  \cite{PasJE21}  introduced the
concept of maximal spectrum for matrices, as comprising eigenvalues with maximum modulus and maximal degeneracy index. Using
Jordan decomposition, he demonstrated that the spectral radius of a
matrix of the form $\sum _i \overline{A_i}\otimes A_i$ lies within
the maximal spectrum. Inspired by Pascoe's approach, we formally
introduce a notion of maximal part of a positive map $\varphi$ with
strictly positive spectral radius, on a finite dimensional
$C^*$-algebra $\mathcal{A}$. It is defined as the positive
map\begin{equation}\label{maximal_part}
	\widehat{\varphi}=\lim\limits_{N\to\infty}\frac{1}{N}\sum\limits_{n=1}^N
	\frac{\varphi^n}{\lVert\varphi^n\rVert_{\beta}},\end{equation} where
$\beta$ is a Jordan basis of $\varphi$. Analyzing this map, we see
that (Theorem \ref{outer:thm7})  the spectral radius is in the
maximal spectrum of $\varphi $, that is, it has Jordan blocks of
maximum size among eigenvalues of the same modulus (see Definition
\ref{maximal spectrum}). Next in Theorem \ref{outer:thm6}, we
observe  that $\widehat{\varphi }$ has remarkable commutation
properties with $\varphi $ and $\widehat{\varphi}(1)$ serves as a
Perron-Frobenius eigenvector of $\varphi$. This is significantly
different from the work Evans and H\o egh-Krohn \cite{EvaDE:HoeR78},
as we now have got  a constructive method to obtain the
Perron-Frobenius eigenvector.

In the last Section, we restrict our attention to completely
positive (CP) maps on finite dimensional $C^*$ algebras. We tackle
two issues described below.

Given a $p$-tuple of matrices, say $A_1, A_2, \ldots , A_p$ in $M_m$, the first question is how to get a basis for the algebra (not $*$-algebra) generated by them.  In \cite{PasJE19}, Pascoe
developed an efficient method for the same. We show that his method
has a very natural interpretation on viewing the tuple as Choi-Kraus
coefficients of the CP maps $\tau (X)= \sum _{i=1}^p A_i^*XA_i$ on
$M_m.$ The algebra we are interested in can be constructed as the space
of Choi-Kraus coefficients of some other CP maps constructed out of
$\tau $.  Theorem \ref{outer:thm12} provides an ample supply of such
CP maps. Pascoe \cite{PasJE19} obtained one such map, namely,
$\gamma_1=(1-b\tau)^{-1}-1 $ where $b$ is a scalar such that
$br(\tau)<1.$ In fact, using the finite dimensionality in Theorem
\ref{outer:thm13}, we show that there is a polynomial $p$ of degree
not exceeding $m^2$, with strictly positive scalar coefficients such
that the Choi-Kraus coefficient space of $p(\tau)$ is equal to the (unital)
subalgebra of $M_m$ generated by $A_i$'s.

The second issue is about  obtaining a characterization of
irreducibility of CP maps in terms of their Choi-Kraus coefficients.
It is well-known that a CP map on $M_m$ is irreducible if and only
if its Choi-Kraus coefficients do not admit any non-trivial common
invariant subspace (\cite{FagF:PelR09}, \cite{FarDR96}). This
condition is no longer necessary for irreducibility of CP maps on
general finite dimensional $C^*$-algebras (See Example
\ref{outer:eg3}). To resolve this for a CP map $\tau$ on a unital
$C^*$-subalgebra $\mathcal{A}$ of $M_m$, we introduce a notion of
the  canonical extension $\widetilde{\tau}$ of $\tau$ to $M_m$ and
in Proposition \ref{outer:prop8}, we establish equivalence between
irreducibility of $\tau$ and $\widetilde{\tau} $. Hence, the desired
necessary and sufficient criterion for the irreducibility of $\tau$
can be described in terms of the Choi-Kraus decompositions of
$\widetilde{\tau}$.

Finally, in Theorem \ref{outer:thm9} we study consequences of
irreducibility condition on the maximal part. This is possible as
irreducibility implies that the map has strictly positive spectral
radius (see Remark \ref{outer:rem1}).  The main observation is that
the maximal part of an irreducible CP map on a finite dimensional
$C^*$-algebra has the form: \[X\mapsto \text{trace}(RX)\,L\] with a
strictly positive density matrix $R$ and a strictly positive element
$L$. Two examples are given to illustrate this.

\section{Preliminaries and background material}\label{Outer:sec:prelim}
In this section, we describe the notation and recall the essential background material concerning completely positive maps, Hilbert bimodules and full Fock modules that will be utilized throughout the paper.
\subsection{Completely positive maps}
For the basic theory of $C^*$-algebras, positive and completely
positive (CP) maps, we refer to (\cite{MurGJ90}, \cite{PauV03},
\cite{BhaBVR:BhaT23}). Typically, we denote a map that is just positive by $\varphi$, and a map that is CP by $\tau$. We recall the well-known Choi-Kraus
decomposition (\cite{ChoMD75}, \cite{BhaR07}) of CP maps on matrix
algebras. A linear map $\tau:M_m\to M_m$ is CP if and only if there
exist $A_1,\ldots, A_p\in M_{m}$, for some $p$, such that
\begin{equation}\label{outer:Choi-Kraus}
	\tau(X)=\sum\limits_{i=1}^p A_i^*XA_i, ~~\forall X\in
	M_m.\end{equation}
The matrices $A_i$'s are known as Choi-Kraus
coefficients of the CP map $\tau$. This decomposition for $\tau $
is not unique. However, if
$$\tau (X)= \sum _{j=1}^qB_j^*XB_j, ~~\forall X\in M_m,$$
is another such decomposition, then every $A_i$ is a linear
combination of $B_j$'s and conversely every $B_j$ is a linear
combination of $A_i$'s. In particular,
$$\mathcal {M}_\tau := \mbox{span}\{A_i: 1\leq i\leq p\}$$
is a subspace of $M_m$, independent of the decomposition. The space
$\mathcal{M}_\tau $ is said to be the \emph{Choi-Kraus coefficient
	space} (or, \emph{coefficient space} in short) of $\tau$. Dimension
of $\mathcal{M}_\tau$ is called the Choi rank of the CP map $\tau$.
W. Arveson \cite{ArvW97}  considers the adjoint of this space
(called the metric operator space of $\tau$) and defines a metric on
it, which is not relevant for this article. For CP maps $\tau$ and $\eta$, $\eta $ is said
to be {\em dominated} by $\tau $ if $\tau -\eta $ is also CP. By
considering Choi-Kraus representations of $\eta $ and $\tau -\eta $,
it is clear that if $\eta $ is dominated by $\tau $, then
$\mathcal{M}_{\eta }\subseteq \mathcal{M}_{\tau }.$

We are interested in determining the
algebra  (not the $*$-algebra) generated by elements of $\mathcal
{M}_\tau.$ In this context, the following observation is useful:

\begin{prop}\label{outer:prop2}
	Let $\tau:M_m\to M_m$ be a CP map. Then $A\in M_{m}$ is in the
	metric space $\mathcal{M}_\tau $ if and only if there exists $q>0$
	such that $\alpha_A$ is dominated by $q\tau$, where $\alpha_A(X)=A^*XA$.
\end{prop}
\begin{proof}
	Suppose that $\tau$ has the Choi-Kraus decomposition as in \eqref{outer:Choi-Kraus}. Let us assume that $q\tau-\alpha_A$ is CP for some $q>0$. Then there are matrices $C_1,\ldots, C_r\in M_m$ such that $q\tau(X)-\alpha_A(X)=\sum\limits_{j=1}^r C_j^*XC_j$ for all $X\in M_n$. Let $\sigma$ be the CP map on $M_m$ defined by $\sigma(X)=A^*XA+\sum\limits_{j=1}^r C_j^*XC_j$. Observe that $\sigma$ has two Choi-Kraus expressions,
	\[\sigma(X)=\sum\limits_{i=1}^p (\sqrt{q}A_i)^*X(\sqrt{q}A_i)=A^*XA+\sum\limits_{j=1}^r C_j^*XC_j,~~\forall X\in M_m.\]
	This ensures that $A$ is a linear combination of $A_i$'s, that is, $A\in \mathcal{M}_\tau$.
	
	Conversely, assume that $A\in\mathcal{M}_\tau$. Let $A=\sum\limits_{i=1}^p\lambda_i A_i$ for some scalars $\lambda_i$. Consider the positive matrix $V=(\overline{\lambda_i}{\lambda_j})\in M_n$ and choose $q>\|V\|$. Then,
	\[q\tau(X)-\alpha_A(X)=q\sum\limits_{i=1}^pA_i^*XA_i-A^*XA=\sum\limits_{i,j}(q\delta_{ij}-\overline{\lambda_i}{\lambda_j})A_i^*XA_j,~~\forall X\in M_n,\]
	Since $qI-V=(q\delta_{ij}-\overline{\lambda_i}{\lambda_j})$ is a positive matrix, $q\tau-\alpha_A$ is
	CP.
\end{proof}

For a linear map $\tau$ on $M_m$, we define the Choi matrix of $\tau$ as
\[C_\tau:=\sum\limits_{i,j}E_{ij}\otimes \tau(E_{ij})\in M_m\otimes M_m.\]
Observe that the association $\tau\mapsto C_\tau$ is continuous. By a seminal result of Choi \cite[Theorem 2]{ChoMD75}, a linear map on $M_m$ is CP if and only if its Choi matrix $C_\tau\in M_{m^2}$ is positive. Define a
linear map $\text{vec}:M_m\to\mathbb{C}^{m^2}$ by
$\text{vec}(E_{ij})=e_{i+(j-1)m}$. For $A\in M_m$, $\text{vec}\, A$
is called the vectorization of $A$. If $\tau$ is an elementary CP map $\alpha_A$, then $(i+(j-1)m)^{\text{th}}$ column of the Choi matrix $C_\tau$ is $a_{ij}\overline{\text{vec}(A)}$, where $A=[a_{ij}]_{1\le i,j\le m}$ for all $i,j$. We recall the following
interpretation of coefficient spaces using conjugated Choi matrices.

\begin{prop}\label{outer:choi_matrix}
	Let $\tau$ be a CP map on $M_m$. Then, the vectorization map \textup{vec} establishes a one-to-one correspondence between the coefficient space of $\tau$ and the range of its conjugated Choi matrix.
	In particular, the Choi rank of $\tau$ is same as the rank of the
	Choi matrix $C_\tau$.
\end{prop}
We now recall the notion of irreducible positive maps and a couple
of basic results from  Evans and H\o egh-Krohn
(\cite{EvaDE:HoeR78}). A positive map $\varphi $  on a finite
dimensional $C^*$-algebra $\mathcal{A}$ is said to be irreducible if
there is no non-trivial projection $p\in \mathcal{A}$, such that the
hereditary $C^*$-subalgebra $p\mathcal{A}p$ is invariant under
$\varphi $. Evans and H\o egh-Krohn (\cite[Lemma 2.1]{EvaDE:HoeR78})
show that if $\mathcal{A}$ is a subalgebra of $M_n$, then $\varphi $
is irreducible if and only if $(id+\varphi )^{n-1}$ is strictly
positive, that is, it sends non-zero positive elements to strictly
positive elements.

Irreducible CP maps on the full matrix algebra $M_m$ have several
interesting properties.
\begin{prop}[Theorem 2, \cite{FarDR96}]\label{outer:irreducible}
	Let $\tau :M_m\to M_m$ be a completely positive map with a
	Choi-Kraus decomposition as in \ref{outer:Choi-Kraus}. Then, the
	following statements are equivalent:
	\begin{enumerate}
		\item The positive  map $\tau $ is irreducible;
		\item The matrices $\{A_1, \ldots , A_p\}$ have no non-trivial common invariant
		subspace;
		\item The algebra generated by $\{A_1, \ldots A_p\}$ is whole of
		$M_m.$
	\end{enumerate}
\end{prop}

Evans and H\o egh-Krohn have also proved the following
\textit{Perron-Frobenious type theorem} for  irreducible positive
maps on finite dimensional $C^*$-algebras.
\begin{thm}[Theorem 2.3, \cite{EvaDE:HoeR78}]\label{outer:thm10}
	Let $\varphi$ be a  positive  linear map on a finite dimensional
	$C^*$-algebra. Then the spectral radius $r(\varphi)$ is an
	eigenvalue of $\varphi$ possessing a  positive eigenvector. If
	$\varphi $ is also irreducible, then $r(\varphi )$ is  a simple
	eigenvalue with an invertible positive eigenvector.
\end{thm}

\subsection{Hilbert modules}

Now, we recall the required definitions and results concerning Hilbert modules and their extensions. Following the standard literature, for the foundational theory and general results regarding Hilbert $C^*$-modules, we refer the reader to standard texts such as Lance \cite{LanEC95}, Wegge-Olsen \cite[Chapter 15]{WegNE93} and Manuilov-Troitsky \cite{ManVM:TroEV05}.

\begin{defn}
	Let $\mathcal{A}$ be a unital $C^*$-algebra. A Hilbert
	$\mathcal{A}$--module is a right $\mathcal{A}$-module (compatible
	with scalar multiplication) together with a $\mathcal{A}$-valued
	inner product, which is complete with respect to the associated norm
	$\|\xi\|=\sqrt{\|\langle \xi,\xi\rangle\|}$. For $\xi\in E$, $|\xi|$
	denotes the positive element $\langle \xi,\xi\rangle^{1/2}$.
\end{defn}

Certainly, we can talk about bounded operators on Hilbert
$C^*$-modules. However, unlike bounded maps on Hilbert spaces, they
may not possess an adjoint.

\begin{defn}
	Let $E$, $F$ be Hilbert $\mathcal{A}$--modules. The space of all
	\textit{adjointable} maps from $E$ to $F$, denoted by
	$\mathscr{B}^a(E,F)$ , is the collection of all right $\mathcal{A}$-linear maps
	$t:E\to F$ such that there exists a map $t^*:F\to E$ satisfying
	$\langle \xi,t\eta\rangle=\langle t^*\xi,\eta\rangle$ for all
	$\xi,\eta\in E$.
\end{defn}

It is evident that every element of $\mathscr{B}^a(E,F)$ is bounded
(right) $\mathcal{A}$--linear map.  When $E=F$, we denote
$\mathscr{B}^a(E,F)$ as $\mathscr{B}^a(E)$, constituting a
$C^*$-algebra. 

\subsection{Hilbert bimodules and tensor products}

We now turn to two-sided structures. For the foundational theory of Hilbert bimodules and their tensor products, we refer to the pioneering work of Kasparov \cite{KasGG80}.

\begin{defn}\label{outer:defnHilbertbimodule}
	A Hilbert $\mathcal{A}$--$\mathcal{B}$--bimodule $E$ is a Hilbert
	$\mathcal{B}$--module which possess a non-degenerate left action of
	$\mathcal{A}$ on $E$, that is, there exists a $*$-homomorphism
	$\pi:\mathcal{A}\to\mathscr{B}^a(E)$ such that
	$E=\overline{span}\{\pi(a)\xi:\, a\in\mathcal{A},\xi\in E\}$. For
	simplicity,  denote $\pi(a)\xi$ by $a\xi$, establishing the left
	$\mathcal{A}$--action.
\end{defn}

\begin{rem}
	In the operator algebra literature, the structure described in Definition \ref{outer:defnHilbertbimodule} is frequently referred to as an $\mathcal{A}$--$\mathcal{B}$ correspondence (see, for instance Muhly-Solel \cite{MuhPS:SolB98}). This terminology distinguishes it from an $\mathcal{A}$--$\mathcal{B}$ imprimitivity bimodule, as we do not assume $E$ implements a strong Morita equivalence between the $C^*$-algebras $\mathcal{A}$ and $\mathcal{B}$. For a comprehensive discussion on Morita equivalence, we refer the readers to \cite{RaeI:WilDP98}. Throughout the paper, we adopt the terminology ``Hilbert $\mathcal{A}$--$\mathcal{B}$--bimodule'' in the sense of an $\mathcal{A}$--$\mathcal{B}$ correspondence, following the convention used by Skeide \cite{SkeM00}.
\end{rem}

For Hilbert $\mathcal{A}$--$\mathcal{B}$--bimodule $E$, the set of all
adjointable $\mathcal{A}$--$\mathcal{B}$--linear operators on $E$
forms a $C^*$-algebra and it is denoted by $\mathscr{B}^{a,bil}(E)$.
Analogous to Hilbert spaces, we can define the direct sum of Hilbert
bimodules.

\begin{defn}
	Let $(E_i)_{i\in\mathbb{I}}$ be a family of Hilbert
	$\mathcal{A}$--$\mathcal{B}$--bimodules. The direct sum $E=\bigoplus_{i\in\mathbb{I}}E_i=\{(\xi_i) : \sum_{i\in\mathbb{I}}\langle\xi_i,\xi_i\rangle \text{ converges in } \mathcal{A} {\text{ with respect to the norm topology}}\}$ which forms an
	$\mathcal{A}$--$\mathcal{B}$--bimodule in a natural way. Define the
	inner product $\langle
	(\xi_i),(\eta_i)\rangle=\sum\limits_{i\in\mathbb{I}}\langle\xi_i,\eta_i\rangle$
	which turns $E$ into a Hilbert $\mathcal{A}$--$\mathcal{B}$--bimodule.
	This Hilbert bimodule $E$ is said to be direct sum of Hilbert bimodules
	$E_i$'s.
\end{defn}

Now, we recall the construction of the exterior and interior tensor
product of Hilbert bimodules. {For the original constructions and detailed proofs establishing the well-definedness of the exterior and interior tensor products of Hilbert bimodules, we refer the reader to Kasparov \cite[Section 2]{KasGG80} and Lance \cite[Chapter 4]{LanEC95}.} For $C^*$-algebras, we denote by
$\mathcal{A}\otimes\mathcal{B}$ the completion of the algebraic
tensor product $\mathcal{A}\otimes_{\text{alg}}\mathcal{B}$ with
respect to the spatial or minimal $C^*$-norm. Since we are working
with Hilbert $\mathcal{A}$--$\mathcal{B}$--bimodules, where
$\mathcal{A}=\mathcal{B}$, we discuss these concepts in this
specific case only.

\begin{defn}
	Let $E$ and $F$ be Hilbert $\mathcal{A}$--$\mathcal{A}$--bimodules.
	Consider the algebraic tensor product $E\otimes_{\text{alg}}F$ with
	the natural
	$\mathcal{A}\otimes_{\text{alg}}\mathcal{A}$--$\mathcal{A}\otimes_{\text{alg}}\mathcal{A}$--action
	$(\xi\otimes\eta)(a\otimes b)=\xi a\otimes\eta b$ and $(a\otimes
	b)(\xi\otimes\eta)=a\xi\otimes b\eta$. Define
	\[\langle \xi_1\otimes\eta_1,\xi_2\otimes\eta_2\rangle=\langle \eta_1,\eta_2\rangle\otimes\langle\xi_1,\xi_2\rangle\, \text{ for all }\xi_1,\xi_2\in E\text{ and }\eta_1,\eta_2\in F.\]
	This sesquilinear form $\langle\cdot,\cdot\rangle$ makes
	$E\otimes_{\text{alg}}F$ into an inner product
	$\mathcal{A}\otimes_{\text{alg}}\mathcal{A}$--$\mathcal{A}\otimes_{\text{alg}}\mathcal{A}$--bimodule.
	Performing double completion, we obtain $E\otimes F$ as a Hilbert
	$\mathcal{A}\otimes\mathcal{A}$--$\mathcal{A}\otimes\mathcal{A}$--bimodule.
	We call $E\otimes F$ the \textit{exterior tensor product} of $E$ and
	$F$.
\end{defn}

\begin{defn}
	Let $E$ and $F$ be Hilbert $\mathcal{A}$--$\mathcal{A}$--bimodules.
	Consider the algebraic tensor product $E\otimes_{\text{alg}}F$ with
	the natural right $\mathcal{A}$--action
	$(\xi\otimes\eta)a=\xi\otimes\eta a$ and define
	\[\langle \xi_1\otimes\eta_1,\xi_2\otimes\eta_2\rangle=\langle \eta_1,\langle\xi_1,\xi_2\rangle\eta_2\rangle\, \text{ for all }\xi_1,\xi_2\in E\text{ and }\eta_1,\eta_2\in F.\]
	This sesquilinear form $\langle\cdot,\cdot\rangle$ makes
	$E\otimes_{\text{alg}}F$ into a semi-inner product
	$\mathcal{A}$--bimodule. Let $N=\{\zeta\in E\otimes_{\text{alg}}F:\,
	\langle\zeta,\zeta\rangle=0\}$. The \textit{interior tensor product}
	of $E$ and $F$ is defined to be the completion of
	$E\otimes_{\text{alg}} F/N$ and it is denoted by $E\odot F$.
	
	We denote the equivalence class of $\xi \otimes \eta$ in $E \odot F$ by $\xi \odot \eta$. Observe that $E \odot F$ is a Hilbert $\mathcal{A}$--$\mathcal{A}$--bimodule with the natural left $\mathcal{A}$--action $a(\xi \odot \eta) = a\xi \odot \eta$. It is interesting to note that the set $N$ in the previous definition is nothing but the subspace generated by the elements of the form $\xi \otimes a\eta - \xi a \otimes \eta$ for $\xi,\eta\in E$ and $a\in\mathcal{A}$.
\end{defn}

In \cite[Theorem 5.2]{PasWL73}, Paschke has described the structure
of CP maps on unital $C^*$-algebras generalizing the famous
Stinespring's theorem (\cite{StiWF55}). Here we provide a brief
overview of the construction. Let $\tau$ be a CP map on a unital
$C^*$-algebra $\mathcal{A}$. Consider the algebraic tensor product
$\mathcal{A}\otimes_{\text{alg}}\mathcal{A}$. For
$a,b,c,d\in\mathcal{A}$, define $\langle a\otimes b,c\otimes
d\rangle=b^*\tau(a^*c)d$ which forms a semi-inner product on $\mathcal{A}\otimes_{\text{alg}}\mathcal{A}$. Let $N=\{z\in\mathcal{A}\otimes_{\text{alg}}\mathcal{A}:\, \langle z,z\rangle=0\}$ and $E$ be the completion of $\mathcal{A}\otimes_{\text{alg}}\mathcal{A}/N$. Then, $E$ is the Hilbert $\mathcal{A}$--$\mathcal{A}$--bimodule in a natural way. Let $\xi=\mathbf{1}\otimes \mathbf{1}+N$. It follows that $\tau(a)=\langle\xi,a\xi\rangle$ for all $a\in\mathcal{A}$. Moreover, $\xi$ is cyclic, meaning, $E=\overline{span}\, \mathcal{A}\xi\mathcal{A}$. The pair $(E,\xi)$ is referred to as the GNS-construction of $\tau$, with $E$ being the GNS-bimodule of $\tau$. The CP map $\tau$ corresponding to the GNS-construction $(E,\xi)$ is denoted by $\tau_\xi$ and plays a central role in our discussion of joint and outer spectral radii. We record the following two properties of $\tau_\xi$ which will be used frequently: For $\xi,\eta\in E$, $\|\tau_\xi\|=\|\xi\|^2$ and $\tau_{\xi}\circ\tau_\eta=\tau_{\eta\odot\xi}$.

\begin{defn}
	Let $\tau$ be a CP map on unital $C^*$-algebra $\mathcal{A}$ and $E$ be a Hilbert $\mathcal{A}$--$\mathcal{A}$--bimodule with an element $\xi\in E$. We call the pair $(E,\xi)$ to be a \textit{GNS-representation} of $\tau$ if $\tau(a)=\langle \xi,a\xi\rangle$ for all $a\in\mathcal{A}$. It is said to be \textit{minimal} if $E=\overline{span}\, \mathcal{A}\xi\mathcal{A}$.
\end{defn}

Note that the GNS-bimodule in the GNS-construction is minimal. For two
minimal GNS-representations $(E,\xi)$ and $(F,\eta)$ of a CP map
$\tau$, the mapping $\xi\mapsto\eta$ extends as an
$\mathcal{A}$--$\mathcal{A}$--linear unitary transformation from $E$
to $F$. Consequently, the GNS-representation is unique up to
(unitary) isomorphism.

\subsection{Full Fock modules}

We conclude our preliminary discussion with the construction of the full Fock bimodule, originally introduced by Pimsner \cite{PimMV97} and Speicher \cite{SpeR98}. It assumes a pivotal role in our discussion on joint and outer spectral radii and their characterisation. For the detailed theory, construction, and applications of Fock modules, we refer to the paper and thesis of Skeide \cite{SkeM00, SkeM01}.

\begin{defn}
	Let $\mathcal{A}$ be a unital $C^*$-algebra and $E$ be a Hilbert
	$\mathcal{A}$--$\mathcal{A}$--bimodule. The \textit{full Fock bimodule}
	over $E$ is defined to be the Hilbert $\mathcal{A}-\mathcal{A}$-bimodule $\mathcal{F}(E)=\bigoplus_{n=0}^{\infty}E^{\odot n}$, where $E^{\odot 0}=\mathcal{A}$, {with its identity element $1 \in \mathcal{A}$ serving as the vacuum vector of the bimodule.}
\end{defn}

\section{Joint and Outer spectral radii on Hilbert \texorpdfstring{$C^*$}{C*}-bimodules}\label{Outer:sec:specrad}

Before coming to defining spectral radii for subsets of Hilbert $C^*$-bimodules, we prove some results on positive maps.
Elements of Hilbert $C^*$-bimodules generate completely positive maps.
However, in the following subsection, we can deal with general
positive maps. We begin with generalizing some of the Popescu's results
(see \cite[Theorem 5.9]{PopG03}) to positive maps on general
$C^*$-algebras. We also show that spectral radius is a spectral
point for positive maps.

\subsection{Positive maps}\label{outer:sec:positive}
Popescu considered only $w^*$-continuous positive maps on
the algebra $\mathscr{B}(\mathcal{H})$ of all bounded operators on a
Hilbert space and some of his proofs does use $w^*$-continuity. We
need these results for general positive maps on arbitrary
$C^*$-algebras. So we provide the proofs.

To begin with note that,  $\varphi$ is a positive map on a unital
$C^*$-algebra $\mathcal{A}$, then $\varphi^n$ is positive for all
$n\in\mathbb{N}$ and we have $\lVert\varphi^n\rVert=\lVert
\varphi^n(1)\rVert$. Hence, the spectral radius of $\varphi$ is
\begin{equation}
	r(\varphi)=\lim\limits_{n\to\infty}\lVert\varphi^n(1)\rVert^{\frac{1}{n}}.
\end{equation}
Recall that, for $v\in\mathcal{A}$, we denote the map $a\mapsto
v^*av$ by $\alpha_v.$ With this notation we have the following
definition.

\begin{defn}
	A positive map $\varphi$ on $\mathcal{A}$
	is said to be {\em elementarily similar\/} to a positive map $\psi$
	on $\mathcal{A}$ if there exists an invertible $a\in \mathcal{A}$
	such that
	$$\varphi=\alpha_a^{-1}\circ\psi\circ\alpha_a,$$
	where $\alpha _a$ denotes the elementary CP map: $\alpha_a(x)=a^*xa$ for all $x\in \mathcal{A}$.
\end{defn}

Now, we have some of the results of Popescu [Theorem 5.9,
\cite{PopG03}] in our setting.

\begin{thm}\label{outer:thm3}
	Let $\varphi$ be a positive map on a unital $C^*$-algebra
	$\mathcal{A}$ and $s$ be any positive real number. Then the
	following statements are equivalent:
	\begin{enumerate}
		\item $r(\varphi)<s$;
		\item There exists $w\ge 0$ such that $\varphi(w)=s(w-1)$;
		\item There exists a strictly positive $v\in\mathcal{A}$ such that
		$\lVert\alpha_v^{-1}\circ\varphi\circ\alpha_v\rVert<s$;
		\item There exists a positive  map $\psi$ on $\mathcal{A}$ which is
		elementarily similar to $\varphi$ such that $\lVert{\psi}\rVert<s$.
	\end{enumerate}
	Moreover, in this case, the element $w\in\mathcal{A}$ as in point (2) is unique and can be obtained by the formula:
	\[w=\sum_{n\ge 0}(\varphi/s)^n(1).\]
\end{thm}
\begin{proof}
	
	\noindent {(1)} $\Rightarrow$ {(2)}: From $(1)$, we have
	$r(\varphi/s)<1$, which implies
	$\lim_{n\to\infty}\left\lVert\left(\frac{\varphi}{s}\right)^n\right\rVert^{1/n}<1$.
	Therefore, the power series
	$\sum_{n\ge0}\left(\frac{\varphi}{s}\right)^n z^n$ has a radius
	of convergence greater than 1. Hence, the series
	$\sum_{n\ge0}\left(\frac{\varphi}{s}\right)^n$ converges {in norm}. Let
	$w=\sum_{n\ge0}\left(\frac{\varphi}{s}\right)^n(1)$. Since $\varphi^n(1)\ge 0$ for all $n\ge 1$, we have $w\ge 0$. Further,
	\[w-1=\dfrac{\varphi}{s}\left(\sum_{n\ge0}\left(\dfrac{\varphi}{s}\right)^n(1)\right)=\dfrac{\varphi}{s}(w) \implies \varphi(w)=s(w-1).\]
	
	\noindent This proves $(2)$.
	
	\noindent {(2)} $\Rightarrow$ {(3)}: Since $w =1+\frac{\varphi}{s}(w)\ge 1$, $w$ is invertible. Consider
	$v=w^{1/2}$. Then
	\[\lVert{\alpha_v}^{-1}\circ\varphi\circ\alpha_v\rVert=\lVert{\alpha_v}^{-1}\circ\varphi\circ\alpha_v(1)\rVert=\lVert s(1-w^{-1})\rVert<s,\]
	where the last inequality follows from $1-w^{-1}<1$. This proves $(3)$.
	
	\noindent {(3)} $\Rightarrow$ {(4)}: Let us take
	${\psi}=\alpha_v^{-1}\circ\varphi\circ\alpha_v$. Then,
	the implication follows trivially.
	
	\noindent {(4)} $\Rightarrow$ {(1)}:
	Since
	${\psi}$ is elementarily similar to $\varphi$, we have
	${\psi}=\alpha^{-1}\circ\varphi\circ\alpha$ for some
	elementary, invertible, CP map $\alpha$. Then
	\[\lVert{\psi}^n\rVert=\lVert \alpha^{-1}\varphi^n\alpha\rVert\le\lVert\alpha^{-1}\rVert\lVert\alpha\rVert\lVert\varphi^n\rVert\implies \lVert{\psi}^n\rVert^{1/n}\le\lVert\alpha^{-1}\rVert^{1/n}\lVert\alpha\rVert^{1/n}\lVert\varphi^n\rVert^{1/n}.\] Taking the limit as $n\to \infty$, we get $r({\psi})\le r(\varphi)$. Similarly, we can obtain $r(\varphi)\le r({\psi})$. Therefore, $r(\varphi)=r({\psi})\le\lVert{\psi}\rVert<s$, which proves $(1)$.
	
	{When the equivalent statements hold, the element $w$ in point (2) is unique. To establish this, suppose there exists another $w' \in \mathcal{A}$ satisfying $\varphi(w')=s(w'-1)$. 
		This rearranges to $(\mathrm{id} - s^{-1}\varphi)(w') = 1$. Since $r(s^{-1}\varphi) < 1$, the operator 
		$\mathrm{id} - s^{-1}\varphi$ is invertible with inverse $\sum_{n\ge 0}(s^{-1}\varphi)^n$. 
		Therefore, $w' = (\mathrm{id} - s^{-1}\varphi)^{-1}(1) = \sum_{n\ge 0}(s^{-1}\varphi)^n(1) = w$, proving uniqueness.} 
\end{proof}

We have the following immediate corollary. This will be useful in
characterizing outer spectral radius.
\begin{cor}\label{outer:cor3}
	Let $\varphi$ be a positive map on a unital $C^*$-algebra $\mathcal{A}$. Then
	\begin{align}\label{outer:SpectralRadiusChar}
		r(\varphi)&=\inf\{\lVert\psi\rVert:\ \psi\text{ is elementarily similar to }\varphi\} \nonumber \\
		&=\inf\{\lVert\alpha_v^{-1}\circ\varphi\circ\alpha_v\rVert:\ v\in\mathcal{A}_+\text{ is invertible}\}.
	\end{align}
\end{cor}
\begin{proof}
	For any $\psi$ which is elementary similar to $\varphi$, there is an
	elementary, invertible, CP map $\alpha$ on $\mathcal{A}$ such that
	$\psi=\alpha^{-1}\varphi\alpha$. Therefore, $\lVert\psi\rVert\ge
	r(\psi)=r(\alpha^{-1}\varphi\alpha)=r(\varphi)$. Now for any
	$s>r(\varphi)$, Theorem \ref{outer:thm3} guarantees existence of a
	CP map $\psi$ which is elementarily similar to $\varphi$ such that
	$\lVert\psi\rVert<s$. Hence,
	\[r(\varphi)=\inf\{\lVert\psi\rVert:\ \psi\text{ is
		elementarily similar to }\varphi\}.\] Similarly,
	$r(\varphi)=\inf\{\lVert\alpha_v^{-1}\circ\varphi\circ\alpha_v\rVert:\
	v\in\mathcal{A}_+\text{ is invertible}\}$ can be proved using ((1)
	$\Rightarrow$ (3)) part of Theorem \ref{outer:thm3}.
\end{proof}

With the help of the Corollary \ref{outer:cor3}, we can directly
establish the following result, which was originally proved for
entry-wise positive matrices by Wielandt \cite{WieH50} and for
positive maps on $C^*$-algebras by  Friedland (\cite{WieH50}) by
entirely different methods.
\begin{thm}\label{outer:friedlandNew}
	Let $\varphi$ be a positive map on a unital $C^*$-algebra
	$\mathcal{A}$. Then
	\[r(\varphi)=\inf\{r(w^{-1}\varphi(w)):\, w\in\mathcal{A}\text{ is strictly positive}\}.\]
\end{thm}
\begin{proof}
	Since norm of a positive element is equal to its spectral radius and $r(ab)=r(ba)$ for all $a,b\in\mathcal{A}$, for any strictly positive $v\in\mathcal{A}$, we infer that
	\[\|\alpha_v^{-1}\circ\varphi\circ\alpha_v\|=\|\alpha_v^{-1}\circ\varphi\circ\alpha_v(1)\|=\|v^{-1}\varphi(v^2)v^{-1}\|=r(v^{-1}\varphi(v^2)v^{-1})=r(w^{-1}\varphi(w)),\]
	where $w=v^2$.
	Now the result follows by  \eqref{outer:SpectralRadiusChar}.
\end{proof}
It is a natural question as to when the  infimum in
\eqref{outer:SpectralRadiusChar} is attained,  that is, when  can we
find a positive map $\psi$ which is elementarily similar to
$\varphi$ with norm equal to $r(\varphi)$. We give a pair of
equivalent conditions that yield an affirmative answer. We will
revisit this inquiry in Section \ref{Outer:sec:positivemap}.

\begin{thm}\label{outer:thm4}
	Let $\varphi$ be a positive map on a unital $C^*$-algebra $\mathcal{A}$ with $r=r(\varphi)$. Then the following statements are equivalent:
	\begin{enumerate}
		\item There is a positive map $\psi$ on $\mathcal{A}$ which is elementarily similar to $\varphi$ with
		$\lVert\psi\rVert=r$;
		\item There exists a strictly positive $w\in\mathcal{A}$ such that $\varphi(w)\le
		r\,w$;
		\item There is a strictly positive $v\in\mathcal{A}$ such that $\lVert{\alpha_v}^{-1}\circ\varphi\circ\alpha_v\rVert=r$.
	\end{enumerate}
\end{thm}
\begin{proof}
	(1) $\Rightarrow$ (2): Since $\psi$ is elementarily similar to
	$\varphi$ with $\lVert\psi\rVert=r$, there is $a\in\mathcal{A}$ such
	that $\psi={\alpha_a}^{-1}\circ\varphi\circ\alpha_a$. Let $w=a^*a$.
	Clearly, $w\in\mathcal{A}$ is positive and invertible. Further,
	\[r\,w-\varphi(w)=ra^*a-\varphi(a^*a)=a^*(r\,1-{a^*}^{-1}\varphi(a^*a){a}^{-1})a=a^*(r\,1-\psi(1))a.\]
	Now as $\lVert\psi(1)\rVert=\lVert\psi\rVert=r$ and $\psi(1)$ is positive, we have
	\[r\,w-\varphi(w)=a^*(r\,1-\psi(1))a\ge0.\]
	
	(2) $\Rightarrow$ (3):  Suppose there is a positive, invertible
	$w\in\mathcal{A}$ such that $\varphi(w)\le r\,w$. Let
	$v=w^{1/2}\in\mathcal{A}$. Clearly, $v$ is positive, invertible and
	we have
	\[\lVert{\alpha_v}^{-1}\circ\varphi\circ\alpha_v\rVert=\lVert({\alpha_v}^{-1}\circ\varphi\circ\alpha_v)(1)\rVert=\lVert v^{-1}\varphi(v^2)v^{-1}\rVert=\lVert w^{-1/2}\varphi(w)w^{-1/2}\rVert.\]
	Since $\varphi(w)\le r\,w$, $w^{-1/2}\varphi(w)w^{-1/2}\le r\,1$ and hence
	\[r=r(\varphi)=r({\alpha_v}^{-1}\circ\varphi\circ\alpha_v)\le\lVert{\alpha_v}^{-1}\circ\varphi\circ\alpha_v\rVert=\lVert w^{-1/2}\varphi(w)w^{-1/2}\rVert\le r.\]
	This proves $(3)$.
	
	\noindent (3) $\Rightarrow$ (1):
	Let $\psi={\alpha_v}^{-1}\circ\varphi\circ\alpha_v$, then the implication follows trivially.
\end{proof}

We conclude this subsection with a compelling application of Theorem
\ref{outer:thm3}. Recall that due to the Perron-Frobenius theorem, as
generalized by Evans and  H\o egh-Krohn, the spectral radius is an eigenvalue
for positive maps on finite dimensional $C^*$-algebras. A positive
map on an infinite dimensional $C^*$-algebra may not posses any
eigenvalue. For instance, consider the unital $C^*$-algebra
$\mathscr{B}(\ell^2)$ and the positive map defined by the right
shift operator $V$ on $\ell^2$, denoted as $\varphi(X)=VXV^*$ for
all $X\in\mathscr{B}(\ell^2)$. Using Gelfand's formula for the
spectral radius it is not hard to see that the spectral radius of
$\varphi $ is 1. Here, $\varphi $ is injective but $\varphi ^n(X)$
converges to zero in the weak operator topology as $n$ tends to
infinity. Therefore, $\varphi$ has no non-trivial eigenvalue. In
particular, it has no non-trivial fixed point. So, the spectral radius
is not in the point spectrum.

We show that the spectral radius of a positive map on a unital
$C^*$-algebra is an element in the spectrum, though it need not be
an eigenvalue. It can be considered as an infinite dimensional
extension of Perron-Frobenius theorem for positive maps on finite
dimensional $C^*$-algebras (Theorem \ref{outer:thm10}).
\begin{thm}\label{outer:thm1}
	Let $\varphi$ be a positive map on a unital $C^*$-algebra
	$\mathcal{A}$. Then $r(\varphi)\in\sigma(\varphi)$.
\end{thm}
\begin{proof}
	If $r(\varphi)=0$, there is nothing to prove. Therefore, without
	loss of generality, we assume $r(\varphi)=1$. Suppose that $1$ does
	not belong to $\sigma(\varphi)$. Then $1-t\varphi$ is invertible for
	all $t\in[0,1]$. Let us define
	$w_t=(1-t\varphi)^{-1}(1)\in\mathcal{A}$. Since $x\mapsto x^{-1}$ is
	continuous on $\mathcal{G}(\mathcal{A})$, the assignment $t\mapsto
	w_t$ defines a continuous map from $[0,1]\to\mathcal{A}$. For $0\le
	t<1$, $r(t\varphi)=tr(\varphi)= t<1$, hence we have the power series
	expression
	\[w_t=1+(t\varphi)(1)+(t\varphi)^2(1)+\cdots.\]
	This shows that $w_t\ge 0$ and $w_t-1=t\varphi(w_t)$ for all $t\in[0,1)$. Therefore, by continuity, we have $\varphi(w_1)=w_1-1$. Thus, there exists $w_1\ge 0$ in $\mathcal{A}$ such that $\varphi(w_1)=w_1-1$, which is equivalent to $r(\varphi)<1$ by Theorem \ref{outer:thm3}. This contradicts $r(\varphi)=1$.
\end{proof}

\subsection{Subsets of Hilbert \texorpdfstring{$C^*$}{C*}-bimodules}\label{outer:sec:HilbertbimoduleSubsets}

Here we introduce  and study notions of joint spectral radius and
outer spectral radius for subsets of Hilbert $C^*$-bimodules. Let $E$
be a Hilbert $\mathcal{A}$--$\mathcal{A}$--bimodule where $\mathcal{A}$
is a unital $C^*$-algebra. Let $G$ be any bounded subset of $E$.
Notice that $\mathcal{A}$ itself is a Hilbert
$\mathcal{A}$--$\mathcal{A}$--bimodule such that $\|a_1\odot \dots
\odot a_n\| = \|a_1\cdots a_n\|$ for any $a_1,\dots , a_n \in
\mathcal{A}$. This is because
\[ \|a_1 \odot \cdots \odot a_n\|^2 = \| \tau_{a_1\odot \dots \odot a_n} \|= \| \tau_{a_n}\circ \cdots \circ \tau_{a_1} \|= \| \tau_{a_1 \cdots a_n}\| =\|a_1\cdots a_n\|^2.\]
Hence, taking inspiration from the notion of joint spectral radius
as in \eqref{outer:eqjoint}, it is natural to define the joint
spectral radius of any bounded subset $G$ of $E$ using the interior tensor
product as follows.
\begin{defn}\label{Outer:defjointE}
	For any bounded subset $G$ of a Hilbert $\mathcal{A}$--$\mathcal{A}$--bimodule $E$, the {\em joint spectral radius\/} of $\mathbf{G}$,  denoted by $\mathbf{\rho_E(G)}$, is defined by
	\[  \rho_E(G)=\lim\limits_{n\to\infty}\sup\limits_{\xi_1,\ldots,\xi_n\in G}\lVert\xi_1\odot\cdots\odot\xi_n\rVert^{1/n}.
	\]
\end{defn}
Now the question is, `` Why does the limit in the definition
exists?''. To answer this  we observe that for any such bounded
subset $G$ of $E$,  we have the corresponding bounded set
$S:=\mathfrak{T}_G:=\{\tau_\xi:\,\xi\in G\}$ of CP maps in the
unital Banach algebra $\mathscr{B}(\mathcal{A})$, where $\tau _\xi $
is the map $a\mapsto \langle \xi, a\xi\rangle $ on $\mathcal{A}.$
Hence, using the Rota-Strang theory of the joint spectral radius in
particular for this $S \subseteq \mathscr{B}(\mathcal{A})$, we have
\[ \rho(S)=\lim\limits_{n\to\infty}\sup\limits_{\xi_1,\ldots,\xi_n\in S}\lVert\tau_{\xi_n}\circ\cdots\circ\tau_{\xi_1}\rVert^{1/n}.\]   Since for a Hilbert bimodule $E$ with $\xi,\eta\in E$, $\tau_\xi\circ\tau_\eta=\tau_{\eta\odot\xi}$ and $\lVert\tau_\xi\rVert=\lVert\xi\rVert_E^2$, it follows that
\[
\sup\limits_{\xi_1,\ldots,\xi_n\in S}\lVert\tau_{\xi_n}\circ\cdots\circ\tau_{\xi_1}\rVert^{1/n}= \sup\limits_{\xi_1,\ldots,\xi_n\in G}\lVert\xi_1\odot\cdots\odot\xi_n\rVert^{2/n}.
\]
Therefore, $\lim\limits_{n\to\infty}\sup\limits_{\xi_1,\ldots,\xi_n\in G}\lVert\xi_1\odot\cdots\odot\xi_n\rVert^{2/n}$ exists and
\[ \rho_{E}(G)^2=\lim\limits_{n\to\infty}\sup\limits_{\xi_1,\ldots,\xi_n\in G}\lVert\xi_1\odot\cdots\odot\xi_n\rVert^{2/n}= \lim\limits_{n\to\infty}\sup\limits_{\xi_1,\ldots,\xi_n\in S}\lVert\tau_{\xi_n}\circ\cdots\circ\tau_{\xi_1}\rVert^{1/n}=\rho(S).\]
Hence, $\rho_{E}(G)$ is well defined and it is the square root of the joint spectral radius of the subset $\mathfrak{T}_G$ of $\mathcal{B}(\mathcal{A})$, that is, we have
\begin{equation}\label{outer:eq17}
	\rho_E(G)=\sqrt{\rho(\mathfrak{T}_G)}.
\end{equation}
If $G=\{\xi_1,\ldots,\xi_d\}$ consists of finitely many elements, then we denote $\rho_E(G)$ by $\rho_E(\xi_1,\ldots,\xi_d)$.

\begin{rem}
	Given distinct elements $\xi_1,\ldots,\xi_d \in E$, we can talk
	about the joint spectral radius of the set $G=\{\xi_1,\ldots,\xi_d\}
	\subseteq E$ and the joint spectral radius of the singleton set
	$\{(\xi_1,\ldots,\xi_d) \}$ of $E^d$. For a Hilbert bimodule $E^d$
	with $\xi_1,\ldots,\xi_d\in E$, we will use the notation
	$\rho_{E^d}(\xi_1,\ldots,\xi_d)$ instead of
	$\rho_{E^d}((\xi_1,\ldots,\xi_d))$ to denote the joint spectral
	radius of the singleton subset $\{(\xi_1,\ldots,\xi_d) \}$ of $E^d$.
	Notice that, it follows from the Definition \ref{Outer:defjointE}
	that,
	\[\rho_E(\xi_1,\ldots,\xi_d)=\lim\limits_{n\to\infty}\sup\limits_{1\le i_1,\ldots,i_n\le d}\|\xi_{i_1}\odot\cdots\odot\xi_{i_n}\|^{1/n}\]
	and
	\[\rho_{E^d}(\xi_1,\ldots,\xi_d)=\lim\limits_{n\to\infty}\left\|\sum|\xi_{i_1}\odot\cdots\odot\xi_{i_n}|^2\right\|^{1/2n}.\]
	Clearly, the two quantities in general need not be equal. Hence, we
	should be conscious that $\rho_E(\xi_1,\ldots,\xi_d)$ and
	$\rho_{E^d}(\xi_1,\ldots,\xi_d)$ are different. This also justifies
	keeping   the Hilbert bimodule  under consideration as a suffix for
	$\rho .$
	
\end{rem}

Recall the Rota-Strang characterization of joint spectral radius
\eqref{outer:eqjoint2}.  We now establish a similar characterisation
for joint spectral radius of a subset of Hilbert $C^*$-bimodule.

\begin{thm}\label{Outer:thm:Rota}
	Let $\mathcal{A}$ be a unital $C^*$-algebra and $G\subseteq E$ be
	any non-empty bounded subset of a Hilbert
	$\mathcal{A}$--$\mathcal{A}$--bimodule $E$. Let $\mathfrak{N}$ denote
	the family of all norms $\mathscr{N}$ on the full Fock bimodule
	$\mathcal{F}(E)$ which are equivalent to the usual norm and satisfy
	the inequality $\mathscr{N}(X\odot Y)\le \mathscr{N}(X)
	\mathscr{N}(Y)$ for all $X\in E^{\odot n}$ and $Y\in E^{\odot m}$,
	with $m, n\geq 0$. Then
	\begin{equation} \rho_E(G)=\inf\limits_{\mathscr{N}\in\mathfrak{N}}\,\sup\limits_{\xi\in
			G}\,\mathscr{N}(\xi).\end{equation}
\end{thm}
\begin{proof}
	Let $\mathscr{N}\in\mathfrak{N}$ be arbitrary. Then there is  $k>0$
	such that $k^{-1}\|\cdot\|_{\mathcal{F}(E)}\le \mathscr{N}(\cdot)\le
	k\|\cdot\|_{\mathcal{F}(E)}$. For any $n\in\mathbb{N}$ and
	$\xi_1,\ldots,\xi_n\in G$,
	\[\|\xi_1\odot\cdots\odot\xi_n\|^{\frac{1}{n}}\le k^{\frac{1}{n}}\mathscr{N}(\xi_1\odot\cdots\odot\xi_n)^{\frac{1}{n}}\le k^{\frac{1}{n}}(\mathscr{N}(\xi_1)\cdots\mathscr{N}(\xi_n))^{\frac{1}{n}}\le k^{\frac{1}{n}}\sup\limits_{\xi\in G}\mathscr{N}(\xi).\] Therefore,
	\[ \sup\limits_{\xi_1,\ldots,\xi_n\in G}\|\xi_1\odot\cdots\odot\xi_n\|^{\frac{1}{n}}\le k^{\frac{1}{n}}\sup\limits_{\xi\in G}\mathscr{N}(\xi). \]
	Taking limit as $n\to\infty$, it follows from the Definition \ref{Outer:defjointE}, that $\rho_{E}(G)\le\sup\limits_{\xi\in G}\mathscr{N}(\xi)$ for all $\mathscr{N}\in\mathfrak{N}$. Hence, we have
	\[\rho_E(G)\le\inf\limits_{\mathscr{N}\in\mathfrak{N}}\,\sup\limits_{\xi\in G}\,\mathscr{N}(\xi).\]
	Now, from  \eqref{outer:eqjoint2} and \eqref{outer:eq17} it follows that
	\[ \rho_E(G)=\sqrt{\rho(\mathfrak{T}_G)}= \sqrt{\inf\limits_{N\in\mathfrak{N}_{\mathcal{B}(\mathcal{A})}}\,\sup\limits_{\tau \in \mathfrak{T}_G}N(\tau)} =\inf\limits_{N\in\mathfrak{N}_{\mathcal{B}(\mathcal{A})}}\,\sup\limits_{\tau \in \mathfrak{T}_G}\sqrt{N(\tau)},\] where $\mathfrak{T}_G=\{\tau_\xi:\,\xi\in G\}$.
	Let $N\in\mathfrak{N}_{\mathscr{B}(\mathcal{A})}$ be arbitrary. Then it can easily be verified that,
	\[\mathscr{N}_F((X_0,X_1,\ldots))=\sqrt{\sum\limits_{n=0}^\infty N(\tau_{X_n})}\ \text{ for all }X_n\in E^{\odot n}\]
	defines a norm on $\mathcal{F}(E)$. Since $N$ is equivalent to the operator norm on $\mathscr{B}(\mathcal{A})$, there is $j>0$ such that $j^{-1}\|\cdot\|\le N(\cdot)\le j\|\cdot\|$. Thus, for any $(X_n)\in\mathcal{F}(E)$, we have
	\[j^{-1}\|(X_n)\|^2=j^{-1}\sum\limits_{n=0}^\infty \|X_n\|^2=\sum\limits_{n=0}^\infty j^{-1}\|\tau_{X_n}\|\le \sum\limits_{n=0}^\infty N(\tau_{X_n}) = \mathscr{N}_F((X_n))^2.\]
	Similarly, \[ \mathscr{N}_F((X_n))^2\le j\|(X_n)\|^2.\]
	Hence, $\mathscr{N}_F$ is equivalent to the usual norm on $\mathcal{F}(E)$. Now let $X\in E^{\odot n}$ and $Y\in E^{\odot m}$ be arbitrary. Then $X\odot Y\in E^{\odot n+m}$ and
	\[\mathscr{N}_F(X\odot Y)=\sqrt{N(\tau_{X\odot Y})}=\sqrt{N(\tau_Y\circ\tau_X)}\le\sqrt{N(\tau_Y)}\sqrt{N(\tau_X)}=\mathscr{N}_F(X)\mathscr{N}_F(Y).\]
	Therefore, $\mathscr{N}_F$ belongs to the family $\mathfrak{N}$. This shows that
	\[\inf\limits_{\mathscr{N}\in\mathfrak{N}}\,\sup\limits_{\xi\in G}\,\mathscr{N}(\xi)\le\sup\limits_{\xi\in G}\,\mathscr{N}_F(\xi)=\sup\limits_{\xi\in G}\,\sqrt{N(\tau_\xi)},\] for all $N\in\mathfrak{N}_{\mathscr{B}(\mathcal{A})}$. Thus,
	\[\inf\limits_{\mathscr{N}\in\mathfrak{N}}\,\sup\limits_{\xi\in G}\,\mathscr{N}(\xi)\le\inf\limits_{N\in\mathfrak{N}_{\mathscr{B}(\mathcal{A})}}\,\sup\limits_{\xi\in G}\,\sqrt{N(\tau_\xi)}= \rho_{E}(G).
	\]
	Hence the proof.
\end{proof}
Taking inspiration from the notion of outer spectral radius for
tuples of  matrices, we define the same for sequences of elements
of Hilbert $C^*$-bimodules.
\begin{defn}
	Let $G=(\xi_1,\xi_2,\ldots ) $ be a finite or countably infinite
	sequence of elements from a Hilbert $\mathcal{A}$--$\mathcal{A}$--bimodule $E$ such that
	\begin{equation}\label{tauG} \tau_G(a):=\sum\limits_{i=1}^\infty\langle\xi_i,a\xi_i\rangle,
		~~\forall a\in\mathcal{A}\end{equation} defines a CP map on
	$\mathcal{A}$. Then, the  {\em outer spectral radius\/} of G, denoted
	by $\mathbf{\widehat{\rho}_E(G)}$, is defined as
	$\widehat{\rho}_E(G)=\sqrt{r(\tau_G)}$.
\end{defn}
It is to be noted that the outer spectral radius is defined for
sequences, instead of subsets, as repetition of same elements would
make a difference.

It is easy to see that for
$G=( \xi _1, \xi _2, \ldots )$, $\tau _G$ as in \ref{tauG} defines a
CP map if and only if $\sum _{i=1}^\infty |\xi _i|^2$ is convergent
in $\mathcal{A}.$ It is clear that countability of $G$ is a
necessity here. When $G=(\xi_1,\dots , \xi_d)$, we write
$\widehat{\rho}_E(G)$ as $\widehat{\rho}_E(\xi_1,\dots , \xi_d)$.

\begin{prop}\label{outer:prop7}
	Let  $G=(\xi_1,\xi_2,\ldots )$ be a sequence  of elements from a
	Hilbert $\mathcal{A}$--$\mathcal{A}$--bimodule $E$  such that
	$\sum\limits_{i=1}^\infty|\xi_i|^2$ converges.  Then
	\begin{align}\label{outer:eq19}
		\widehat{\rho}_E(G)&=\lim\limits_{n\to\infty}\left\|\sum\limits_{i_1,\ldots,i_n\in\mathbb{N}}|\xi_{i_1}\odot\cdots\odot\xi_{i_n}|^2\right\|^{1/2n}\\
		&=\lim\limits_{n\to\infty}\sup\left\{\left\|\sum\limits_{\xi\in
			\mathscr{F}}|\xi|_{E^{\odot n}}^2\right\|^{1/2}:\
		\mathscr{F}\subseteq \mathscr{P}_n(G)\text{ is finite}\right\}^{1/n}
	\end{align}
\end{prop}
\begin{proof}
	We have $G=(\xi_1,\xi_2,\ldots)$ such that $\sum\limits_{i=1}^\infty|\xi_i|^2$ converges. Clearly, $G$ is associated to a CP map $\tau_G$ on $\mathcal{A}$ defined by $\tau_G(a)=\sum\limits_{i=1}^\infty\langle\xi_i,a\xi_i\rangle$ for all $a\in\mathcal{A}$. Observe that, the given CP map $\tau_G$ can be expressed as $\tau_G=\sum\limits_{i=1}^\infty\tau_{\xi_i}$. Then we have
	\[\widehat{\rho}(G)^2=r(\tau_G)=\lim\limits_{n\to\infty}\|\tau_G^n(1)\|^{1/n}=\left\|\sum\limits_{i_1,\ldots,i_n\in\mathbb{N}}|\xi_{i_1}\odot\cdots\odot\xi_{i_n}|^2\right\|^{1/n}.\]
	Hence the proof.
\end{proof}

The expression in \eqref{outer:eq19} provides a Gelfand-type formula for the outer spectral radius of $G$. Alternatively, we can obtain characterizations of this quantity in the spirit of the Wielandt-Friedland formula by evaluating the associated completely positive map on strictly positive elements.
\begin{prop}\label{outer:prop4}
	Let $G=(\xi_1,\xi_2,\ldots)$ be a sequence of elements from a Hilbert $\mathcal{A}$--$\mathcal{A}$--bimodule $E$ such that $\sum\limits_{i=1}^\infty|\xi_i|^2$ converges. Then,
	\begin{align}\label{outer:friedland-wielandt}
		\widehat{\rho}_E( G)&=\inf\left\{\left\|\sum\limits_{i=1}^\infty|v\xi_i v^{-1}|^2\right\|^{1/2}:\, v\in\mathcal{A}\text{ is strictly positive}\right\} \nonumber\\
		&=\inf\left\{r\left(\sum\limits_{i=1}^\infty\langle w\xi_i w^{-1},\xi_i\rangle\right)^{1/2}:\, w\in\mathcal{A}\text{ is strictly positive}\right\}.
	\end{align}
\end{prop}
\begin{proof}
	Given a sequence $G=(\xi_1,\xi_2,\ldots)\in \bigoplus_{\mathbb{N}} E$, we utilize the identity \eqref{outer:SpectralRadiusChar} to obtain the first expression for the outer spectral radius of $G$. Specifically, we have     
	\[         
	\widehat{\rho}_E(G) = \sqrt{r(\tau_G)} = \inf\left\{\lVert\alpha_v^{-1}\circ\tau_G\circ\alpha_v\rVert^{1/2}:\ v\in\mathcal{A}_+\text{ is invertible}\right\}.     
	\]     
	Since the norm of a positive map is attained at the identity element, we evaluate 
	\[\alpha_v^{-1}\circ\tau_G\circ\alpha_v(1) = \sum_{i=1}^\infty \langle v\xi_i v^{-1}, v\xi_i v^{-1}\rangle = \sum_{i=1}^\infty |v\xi_i v^{-1}|^2.\]
	Substituting this into the infimum yields the first equality.       
	To establish the second characterisation, we apply Friedland's result (Theorem \ref{outer:friedlandNew}) to the CP map $\tau_G$. For any strictly positive $w \in \mathcal{A}$, we evaluate the expression at $w$ to find 
	\[w^{-1}\tau_G(w) = \sum_{i=1}^\infty w^{-1}\langle \xi_i,w\xi_i\rangle = \sum_{i=1}^\infty \langle w\xi_i w^{-1},\xi_i\rangle.\] Taking the spectral radius and the square root completes the proof. 
\end{proof}

The following Proposition compares the joint spectral radius with
the outer spectral radius.

\begin{prop}\label{outer:prop1}
	Let $E$ be a Hilbert $\mathcal{A}$--$\mathcal{A}$--bimodule. Then, for any $G=(\xi_1, \xi_2,\dots )$ satisfying $\sum\limits_{i=1}^\infty|\xi_i|^2< \infty$, we have $\rho_{E}(G) \leq \widehat{\rho}_E(G)$ and in particular if $G= (\xi_1,\ldots,\xi_d)$,
	\[d^{-1/2}\widehat{\rho}_E(\xi_1,\ldots,\xi_d)\le\rho_E(\xi_1,\ldots,\xi_d)\le\widehat{\rho}_E(\xi_1,\ldots,\xi_d).\]
\end{prop}
\begin{proof}
	From \eqref{outer:eq19} and the definition of joint spectral radius, we observe that $\rho_E(G)\le \widehat{\rho}_E(G)$. Further,
	it follows from the Proposition \ref{outer:prop7} that,
	\begin{align*}
		\widehat{\rho}_E(\xi_1,\ldots,\xi_d)&=\lim\limits_{n\to\infty}\left\|\sum\limits_{1\le i_1,\ldots,i_n\le d}|\xi_{i_1}\odot\cdots\odot\xi_{i_n}|^2\right\|^{1/2n}\\
		&\le \sqrt{d}\lim\limits_{n\to\infty}\sup\limits_{1\le i_1,\ldots,i_n\le d}|\xi_{i_1}\odot\cdots\odot\xi_{i_n}|^{1/n}\\
		&=\sqrt{d}\,\rho_E(\xi_1,\ldots,\xi_d).
	\end{align*}
	This completes the proof.
\end{proof}

The following lemma is useful in proving an approximation theorem
for the joint spectral radius.
\begin{lem}\label{outer:lem2}
	Let $E$ be a Hilbert $\mathcal{A}$--$\mathcal{A}$--bimodule with a bounded set $ G\subset E$. Then for each $k\in\mathbb{N}$, $\rho_{E^{\otimes k}}( G^{(k)})=\rho_E( G)^k$ where $ G^{(k)}=\{\xi^{\otimes k}:\, \xi\in G\}$ is a bounded subset of Hilbert $\mathcal{A}^{\otimes k}$--$\mathcal{A}^{\otimes k}$--bimodule $E^{\otimes k}$.
\end{lem}
\begin{proof}
	For any $\xi_1,\ldots,\xi_n\in E$, we have
	\[\|\xi_1^{\otimes k}\odot\cdots\odot\xi_n^{\otimes k}\|=\|\tau_{\xi_1^{\otimes k}\odot\cdots\odot\xi_n^{\otimes k}}(1\otimes\cdots\otimes 1)\|^{1/2}=\|(\tau_{\xi_1\odot\cdots\odot\xi_n}(1))^{\otimes k}\|^{1/2}=\|\xi_1\odot\cdots\odot\xi_n\|^k.\]
	Therefore, we infer that
	\[\rho_{E^{\otimes k}}( G^{(k)})=\lim\limits_{n\to\infty}\sup\limits_{\xi_1,\ldots,\xi_n\in G}\|\xi_1^{\otimes k}\odot\cdots\odot\xi_n^{\otimes k}\|^{1/n}=\lim\limits_{n\to\infty}\sup\limits_{\xi_1,\ldots,\xi_n\in G}\|\xi_1\odot\cdots\odot\xi_n\|^{k/n}=\rho_E( G)^k.\]
\end{proof}
In particular when $G$ is finite, we can approximate the joint spectral radius by the outer spectral radius as follows.
\begin{thm}\label{outer:thm22}
	Let $E$ be a Hilbert $\mathcal{A}$--$\mathcal{A}$--bimodule with $\xi_1,\ldots,\xi_d\in E$. Then,
	\[d^{-1/2k}\widehat{\rho}_{E^{\otimes k}}(\xi_1^{\otimes k},\ldots,\xi_d^{\otimes k})^{1/k}
	\le\rho_E(\xi_1,\ldots,\xi_d)\le\widehat{\rho}_{E^{\otimes k}}(\xi_1^{\otimes k},\ldots,\xi_d^{\otimes k})^{1/k},\]
	for every $k\geq 1$ and
	\begin{equation}\label{formula}
		\rho_E(\xi_1,\ldots,\xi_d)= \lim\limits_{k\to\infty}\widehat{\rho}_{E^{\otimes k}}(\xi_1^{\otimes k},\ldots,\xi_d^{\otimes
			k})^{1/k}.\end{equation}
\end{thm}
\begin{proof}
	By Lemma \ref{outer:lem2}, $\rho_{E^{\otimes k}}(\xi_1^{\otimes
		k},\ldots,\xi_d^{\otimes k})=\rho_E(\xi_1,\ldots,\xi_d)^k$. Now,
	using Proposition \ref{outer:prop1}, we have
	$d^{-1/2}\widehat{\rho}_{E^{\otimes k}}(\xi_1^{\otimes k},\ldots,\xi_d^{\otimes
		k})\le \rho_E(\xi_1,\ldots,\xi_d)^k\le \widehat{\rho}_{E^{\otimes k}}(\xi_1^{\otimes
		k},\ldots,\xi_d^{\otimes k})$. This proves the inequality and taking
	limit  as $k\to\infty$, we get our desired approximation of the
	joint spectral radius.
	
\end{proof}

We explain the notions described above  in the particular case, where
the Hilbert $C^*$-bimodule $E$ is equal to $\mathcal{A}$, with
standard inner product: $\langle a, b\rangle=a^*b$ for $a,b$ in
$\mathcal{A}.$ Let $\mathcal{A}$ be a unital $C^*$-algebra with
$a_1,\ldots,a_d\in\mathcal{A}$. Consider the Hilbert
$\mathcal{A}$--$\mathcal{A}$--bimodule $E=\mathcal{A}$ and let $G=\{
a_1,\dots , a_d\}$ be any subset of $\mathcal{A}$. Clearly,
\[ \|a_1 \odot \dots \odot a_n\|^2 = \| \tau_{a_1\odot \dots \odot a_n} \|= \| \tau_{a_n}\circ \cdots \circ \tau_{a_1} \|= \| \tau_{a_1 \cdots a_n}\| =\|a_1\cdots a_n\|^2.\]
Hence, the Definition \ref{Outer:defjointE} gives us that,
\begin{equation}\label{outer:eq8}
	\rho_E(G)=\rho_{\mathcal{A}}(a_1,\ldots,a_d)=\lim\limits_{n\to\infty}\sup\limits_{1\le i_1,\ldots,i_n\le d}\lVert a_{i_1}\cdots a_{i_n}\rVert^{1/n},
\end{equation}
which is the familiar joint spectral radius as in \eqref{outer:eqjoint}.
For $G=\{a_1,\ldots,a_d\}\subseteq\mathcal{A}$,
$\tau_G:\mathcal{A}\to \mathcal{A}$ will be of the form
\begin{equation}
	\tau_G(x)=\sum\limits_{i=1}^d a_i^* xa_i.
\end{equation}
Observe that, by induction we obtain
\[\tau_{G}^n(x)=\sum\limits_{1\le i_1,\ldots,i_n\le d}(a_{i_1}\cdots a_{i_n})^*x(a_{i_1}\cdots a_{i_n})=\sum\limits_I a_I^* x a_I.\]
Here, the multi-index $I$ is running over the set $\{(i_1,\ldots,i_n):1\le i_1,\ldots,i_n\le d\}$ and for $I=(i_1,\ldots,i_n)$, we denote $a_I=a_{i_1}a_{i_2}\cdots a_{i_n}$. Therefore, $\lVert\tau_G^n\rVert=\lVert\tau_G^n({1})\rVert=\left\lVert\sum\limits_I a_I^* a_I\right\rVert$.
Hence, the outer spectral radius of $G$ has the expression
\begin{equation}\label{outer:eq9}
	\widehat{\rho}_{\mathcal{A}}(G)=\sqrt{r(\tau_G)}=\lim\limits_{n\to\infty}\lVert\tau_G^n\rVert^{1/2n}=\lim\limits_{n\to\infty}\left\lVert\sum\limits_{I}a_I^*a_I\right\rVert^{1/2n}.
\end{equation}

\begin{thm} Let
	$a_1, a_2, \ldots , a_d$ be  elements in a unital $C^*$-algebra
	$\mathcal{A}.$ Then,
	\begin{eqnarray}
		\widehat{\rho}_{\mathcal{A}}(a_1,\ldots,a_d)
		&=&\inf\{\lVert(va_1v^{-1},\ldots,va_dv^{-1})\rVert:\ v\in\mathcal{A}\text{ is strictly positive}\}\\
		&=&\inf\left\{\left\|\sum\limits_{i=1}^d|v a_i
		v^{-1}|^2\right\|^{1/2}:\, v\in\mathcal{A}\text{ is strictly
			positive}\right\},\\
		&=& \inf\left\{ r\left(\sum_{i=1}^d (a_i v^{-1})^*(v
		a_i)\right)^{1/2} : \, v \in \mathcal{A} \text{ is strictly
			positive} \right\}.
	\end{eqnarray}
	Furthermore,
	\begin{equation}
		\rho_{\mathcal{A}}(a_1,\ldots,a_d)=
		\lim\limits_{k\to\infty}\widehat{\rho}_{\mathcal{A}^{\otimes k}}(a_1^{\otimes
			k},\ldots,a_d^{\otimes k})^{1/k}.
	\end{equation}
\end{thm}

\begin{proof}
	Take $E=\mathcal{A}$ in  Proposition \ref{outer:prop4} and Theorem \ref{outer:thm22}.
\end{proof}

\begin{rem}
	For a tuple of $m\times m$ matrices $(A_1,\ldots,A_d)$,
	$$\widehat{\rho}_{M_m}(A_1,\ldots,A_d)=\sqrt{r\left(\sum\limits_{i=1}^d\overline{A_i}\otimes A_i\right)}=\sqrt{r\left(\sum\limits_{i=1}^d\overline{A_i^*}\otimes A_i^*\right)}=\widehat{\rho}_{M_m}(A_1^*,\ldots,A_d^*).$$
	Furthermore, if $\mathcal{A}$ is a finite dimensional $C^*$-algebra,
	then for a tuple $(a_1,a_2,\ldots , a_d)$ of elements of
	$\mathcal{A}$,
	$\widehat{\rho}_{\mathcal{A}}(a_1,\ldots,a_d)=\widehat{\rho}_{\mathcal{A}}(a_1^*,\ldots,a_d^*)$.
	However, this equality does not generally hold if $\mathcal{A}$ is
	not finite dimensional. For example, consider the unital
	$C^*$-algebra $\mathcal{A}=\mathscr{B}(\ell^2)$ and the isometries
	$V,W$ on $\ell^2$ defined as $V(e_n)=e_{2n+1}$ and $W(e_n)=e_{2n}$
	on the standard basis. Let $\tau(X)=V^*XV+W^*XW$ and
	$\sigma(X)=VXV^*+WXW^*$. Then,
	$\widehat{\rho}_{\mathcal{A}}(V,W)=\sqrt{r(\tau)}=\sqrt{2}$, whereas
	$\widehat{\rho}_{\mathcal{A}}(V^*,W^*)=\sqrt{r(\sigma)}=1$.
\end{rem}

\section{Positive maps on finite dimensional algebras}\label{Outer:sec:positivemap}

For the rest of our discussion, we focus on  finite dimensional
$C^*$-algebras. To begin with we look at general positive maps on
finite dimensional algebras. We study the maximal part of a general
positive map, and its analysis allows us to construct a positive
eigenvector with Perron-Frobenius eigenvalue.

Let $\mathcal{A}$ be a unital $C^*$-algebra with $a\in\mathcal{A}$.
Consider the elementary CP map $\alpha_a(x)=a^*xa, ~x\in
\mathcal{A}$. Applying Corollary \ref{outer:cor3}, we obtain the following well-known formula for
the spectral radius $r(a)$ of $a$:
\begin{align}\label{outer:eq10}
	r(a)&=\inf\{\lVert b\rVert:\ b\in\mathcal{A}\text{ is similar to }a\}\\
	&=\inf\{\lVert vav^{-1}\rVert:\ v\in\mathcal{A}_+\text{ is invertible}\}.
\end{align}
It is a natural question as to when the infimum in
\eqref{outer:eq10} can be attained. With help from the following
Lemma, we show that in finite dimensions the boundedness of the set $\{
(a/r(a))^n: n\geq 0\}$ ensures this.
\begin{lem}\label{outer:lem10}
	Let $A\in M_m$ be a matrix with $r(A)=1$. If there is $M>0$ such that $\lVert A^n\rVert\le M$ for all $n\in\mathbb{N}$, then there exists $P\in GL_m$ such that $\lVert PAP^{-1}\rVert=1$.
\end{lem}
\begin{proof}
	Let $J$ be the Jordan form of the matrix $A$. Then $A=SJS^{-1}$ for some $S\in GL_{m}$. Let $\lambda\in\mathbb{C}$ be an eigenvalue
	of $A$ and $J_{\lambda}$ be the Jordan block corresponding to
	$\lambda$ of size $k$. First, suppose $|\lambda|=1$. Note that,
	$J_{\lambda}=\lambda I_k+N$ where $N$ is the $k\times k$ nilpotent
	matrix with nilpotency $k$. Observe that,
	\[J_{\lambda}^n=\lambda^n
	I_k+\binom{n}{1}\lambda^{n-1}N+\binom{n}{2}\lambda^{n-2}N^2+\cdots+\binom{n}{k-1}\lambda^{n-k+1}N^{k-1}.\]
	When $k\geq 2$, since $\|A^n\|\leq M$, we have
	\[n=|n\lambda^{n-1}|=|\langle e_1,J_{\lambda}^n e_2\rangle|\le\lVert J_{\lambda}^n\rVert\le\lVert J^n\rVert\le\lVert S\rVert\lVert S^{-1}\rVert\lVert A^n\rVert\le \lVert S\rVert\lVert S^{-1}\rVert M\ \text{ for all }n,\]
	which is paradoxical. So, we must have $k=1$ and hence for $|\lambda|=1$, $J_{\lambda}$ is of size 1. Thus, the Jordan form of $A$ looks like $J=\left[\begin{smallmatrix}
		D &  & & & \\
		& J_1 & & & \\
		& & J_2 & & \\
		& & & \ddots & \\
		& & & & J_l
	\end{smallmatrix}\right]$ where $D$ is a diagonal matrix with diagonal entries of modulus $1$ and each $J_i$ is a Jordan block corresponding to eigenvalue of modulus less than $1$.
	
	Now, consider $|\lambda|<1$. Choose $\epsilon>0$ such that $|\lambda|+(k-1)\epsilon<1$ and let
	$C=\text{diag}(c_1,\ldots,c_k)$ where $c_j=\epsilon^{k-j+1}\ \forall j\in\{1,\ldots,k\}.$ We infer that
	\[\lVert CJ_{\lambda}C^{-1}\rVert=\left\lVert\lambda I_k+\epsilon\sum\limits_{j=2}^kE_{(j-1)j}\right\rVert\le|\lambda|+(k-1)\epsilon<1.\]
	where $\{E_{ij}: 1\leq i,j \leq k\}$ is the standard basis of $M_k$ consisting of elementary matrices. Define \[T=\left[\begin{smallmatrix}
		I &  & & & \\
		& C_1 & & & \\
		& & C_2 & & \\
		& & & \ddots & \\
		& & & & C_l
	\end{smallmatrix}\right]\text{ and }P=TS^{-1}.\]
	Then, $\lVert PAP^{-1}\rVert=\lVert TJT^{-1}\rVert=\max\{\lVert D\rVert,\lVert C_jJ_j {C_j}^{-1}\rVert:\ j\in\{1,\ldots,l\}\}=1$.
\end{proof}
We now answer the above query in form of the following theorem.
\begin{thm}\label{outer:thm15}
	Let $\mathcal{A}$ be a finite dimensional $C^*$-algebra and $a\in\mathcal{A}$ with $r=r(a)>0$. Then the following statements are equivalent:
	\begin{enumerate}
		\item There is $M>0$ such that $\left\lVert\frac{a^n}{r^n}\right\rVert\le M$ for all $n\in\mathbb{N}$.
		\item There is $b\in \mathcal{A}$ which is similar to $a$ and $\lVert b\rVert=r$.
		\item There is strictly positive $v\in\mathcal{A}$ such that $\lVert v^{-1}av\rVert=r$.
	\end{enumerate}
\end{thm}
\begin{proof}
	As $\mathcal{A}$ is a finite dimensional $C^*$-algebra, by Structure Theorem, we have $\mathcal{A}=M_{n_1}\oplus\cdots\oplus M_{n_d}$. Therefore, for given $a\in\mathcal{A}$, $a$ is of the form $a=X_1\oplus\cdots\oplus X_d$ where each $X_i\in M_{n_i}$. Without loss of generality, we assume that $r=1$.
	
	\noindent {(1)} $\Rightarrow$ {(2)}: Suppose, for given $a\in\mathcal{A}$ with $r(a)=1$, there is $M>0$ such that $\lVert a^n\rVert\le M$ for all $n\in\mathbb{N}$. Let $I=\{i:\, r(X_i)=1, 1\le i\le d\}$ and $K=\{i:\, r(X_i)<1, 1\le i\le d\}$. Since $r=\max\{r(X_i):\ i\in\{1,\ldots,d\}\}=1$, we have $I\ne\emptyset$ and $I\cup K=\{1,\ldots,d\}$. By Corollary \ref{outer:cor3},
	\[
	\forall\, i\in K,\ \exists\, P_i\in GL_{n_i}\text{ such that }\lVert P_iX_i{P_i}^{-1}\rVert<1
	\]
	Now, for any $i\in I$, spectral radius of $X_i$ is $1$ and $\lVert X_i^n\rVert\le\lVert a^n\rVert\le M$ for all $n\in\mathbb{N}$. Then by Lemma \ref{outer:lem10}
	\[
	\exists\, P_i\in GL_{n_i} \text{ such that }\lVert P_iX_i{P_i}^{-1}\rVert=1.
	\]
	Define $$p=\left[\begin{smallmatrix}
		P_1 &  & &  \\
		& P_2 & & \\
		& & \ddots & \\
		& & & P_d
	\end{smallmatrix}\right]\text{ and }b=pap^{-1}\in\mathcal{A}.$$
	Then, $\lVert b\rVert=\max\{\lVert P_iX_i{P_i}^{-1}\rVert:\ i\in\{1,\ldots,d\}\}=1$.
	
	\noindent{(2)} $\Rightarrow$ {(3)}: Assume that there is $b\in \mathcal{A}$ which is similar to $a$ with $\|b\|=1$. Thus, there is an invertible $p\in\mathcal{A}$ such that $\|b\|=\lVert pap^{-1}\rVert=1$. Let $\tau$ be the elementary CP map $\alpha_a$ and $w=p^*p$. Then
	\[w-\tau(w)=p^*p-a^*p^*pa=p^*(1-{p^*}^{-1}a^*p^*pa{p}^{-1})p=p^*(1-(pap^{-1})^*(pap^{-1}))p\ge0\]
	Therefore, by Theorem \ref{outer:thm4}, there exists positive, invertible $v\in\mathcal{A}$ such that $\lVert\alpha_v\circ\tau\circ\alpha_v^{-1}\rVert=1$.
	Clearly,
	\[ \lVert\alpha_v\circ\tau\circ\alpha_v^{-1}\rVert=\lVert\alpha_{v^{-1}av}\rVert=\lVert v^{-1}av\rVert^2.\]
	Thus, there is a positive, invertible $v\in \mathcal{A}$ such that $\lVert v^{-1}av\rVert=1$.

	\noindent{(3)} $\Rightarrow$ {(1)}: Let $b=v^{-1}av$ and we have $\lVert b\rVert=1$. Then for all $n\in\mathbb{N}$,
	\[\lVert a^n\rVert=\lVert vb^n v^{-1}\rVert\le \lVert v\rVert\lVert v^{-1}\rVert\lVert b\rVert^n=\lVert v\rVert\lVert v^{-1}\rVert.\]
	This proves $(1)$.
\end{proof}
For a CP map $\tau$ on finite dimensional $C^*$-algebra,  it is evident that boundedness of the sequence $\left\lVert\frac{\tau^n}{r(\tau)^n}\right\rVert$ is a necessary condition for existence of an elementarily similar CP map $\sigma$ with $\|\sigma\|=r(\tau)$. But it is not sufficient, as evidenced by the following example.
\begin{eg}\label{outer:eg1}
	Let $\tau:M_2\to M_2$ be a CP map defined by
	\[\tau\left(\left[\begin{matrix}
		a & b \\
		c & d
	\end{matrix}\right]\right)=\left[\begin{matrix}
		a+d & 0 \\
		0 & 0
	\end{matrix}\right].\]
	It's easy to see that $\tau^n(I)=\begin{bmatrix}
		2 & 0\\
		0 & 0
	\end{bmatrix}$. Therefore, $\lVert\tau^n\rVert=\lVert\tau^n(I)\rVert=2$ for all $n\in\mathbb{N}$. This shows that $r(\tau)=1$ and $\left\lVert\frac{\tau^n}{r(\tau)^n}\right\rVert$ is bounded. Now, if there is a CP map $\widetilde{\tau}$ which is elementarily similar to $\tau$ such that $\lVert\widetilde{\tau}\rVert=r(\tau)$, then due to Theorem \ref{outer:thm4}, there is a strictly positive $V=\left[\begin{matrix}
		a & b \\
		c & d
	\end{matrix}\right]\in M_2$ such that $\tau(V)\le V$. But
	\[\begin{bmatrix}
		a+d & 0\\
		0 & 0
	\end{bmatrix}\le \begin{bmatrix}
		a & b\\
		c & d
	\end{bmatrix}\implies\begin{bmatrix}
		-d & b\\
		c & d
	\end{bmatrix}\ge 0\implies d\le0,\]
	which contradicts the strict positivity of $d$.
\end{eg}

In their seminal paper \cite{EvaDE:HoeR78},  Evans and H\o egh-Krohn explored Perron-Frobenius theory for positive maps on $C^*$-algebras. They showed that for any positive map $\varphi$ on a finite dimensional $C^*$-algebra $\mathcal{A}$, the spectral radius $r(\varphi )$ is an eigenvalue of $\varphi $ with a positive eigenvector. However, their proof is completely existential.  We adopt a different approach and
explicitly construct the positive eigenvector.

Note that when $r(\varphi)=0$, there is $n$ such that $\varphi^n=0$
and $\varphi^{n-1}\ne0$, hence $\varphi^{n-1}(1)$ is a positive
eigenvector with eigenvalue equal to the spectral radius. Therefore,
in the following we assume $r(\varphi)>0$.  The key to our
construction is the `maximal part of a positive map' defined as
follows. This map was introduced by Pascoe \cite{PasJE21} using matrices. The justification for the
nomenclature will be clear from the discussion below.

\begin{defn}
	Let $\varphi$ be a positive map on a finite dimensional
	$C^*$-algebra $\mathcal{A}$ with $r(\varphi)>0$. The {\em maximal
		part\/} of $\varphi$ is defined as
	\begin{equation}\label{outer:eq4}
		\widehat{\varphi}=\lim\limits_{N\to\infty}\frac{1}{N}\sum\limits_{n=1}^N\frac{\varphi^n}{\lVert\varphi^n\rVert_{\beta}},
	\end{equation}
	where $\beta$ is a Jordan basis of $\varphi$ and
	$\|\varphi^n\|_\beta:=\|[\varphi^n]_\beta\|_\infty$.
\end{defn}

The series in equation \eqref{outer:eq4} is not well-defined if
$r(\varphi )$ equals 0,  as $\varphi$ would then become nilpotent. For this reason, we have assumed $r(\varphi)>0$. Now, we need
to show that the series in \eqref{outer:eq4} is convergent. Before
doing that, let us explain the norm appearing in the denominator.
Suppose $m:=\text{dim}(\mathcal{A})$, and $[\varphi]_{\gamma}\in
M_m$ is the matrix of $\varphi$ with respect to a basis  $\gamma$ of
$\mathcal{A}.$  Now  $\|\cdot\|_\infty$ denotes the supremum norm on
$M_m$, that is,
$$\|(a_{ij})\|_\infty=\sup\{|a_{ij}|:\,i,j\in\{1,\ldots, m\}\}.$$
As $\mathcal{A}$ is a finite dimensional $C^*$-algebra,
$\lVert\cdot\rVert_{\gamma}$ is equivalent to the usual operator
norm $\lVert\cdot\rVert$ of the Banach algebra
$\mathscr{B}(\mathcal{A})$. Currently, we are choosing Jordan basis
of $\varphi $, meaning that on this basis the matrix of $\varphi $
is a direct sum of Jordan blocks. By uniqueness of the Jordan
decomposition the norm
$\|\varphi^n\|_\beta:=\|[\varphi^n]_\beta\|_\infty$ is independent
of the choice of the Jordan basis. Consequently, the maximal part of
$\varphi$ is also independent of the choice of Jordan basis.

Let $x\in\mathcal{A}$  be a vector in a Jordan basis  $\beta $ of
$\varphi $. Suppose that it is a generalized eigenvector of $\varphi$ of
rank $k$ corresponding to eigenvalue $\lambda$.  Then, taking powers
of the Jordan block,  we see that  for any $n\in\mathbb{N}$,

\begin{equation}\label{powers}\varphi^n(x)=\sum\limits_{j=0}^{k-1}
	\binom{n}{j}\lambda^{n-j}.(\varphi-\lambda)^{j}x,\end{equation} and
$(\varphi-\lambda)^{j}x$ are Jordan basis vectors. Using this
expression, it is possible to compute $\|\varphi ^n\|_\beta .$ Here
it is convenient to have the following definition.

\begin{defn}\label{maximal spectrum}
	Let $\varphi$ be a positive map on a finite dimensional
	$C^*$-algebra. The degeneracy index $d_{\varphi,\lambda}$ of an
	eigenvalue $\lambda$ is defined to be the size of the largest Jordan
	block corresponding to $\lambda$, or equivalently, the power of the
	$(z-\lambda)$  in the factorization of the minimal polynomial $m(z)$
	of $\varphi$. The {\em maximal degeneracy index\/} of $\varphi$,
	denoted by $d_\varphi$, is defined as
	\[d_\varphi=\max\{d_{\varphi,\lambda}:\, |\lambda|=r(\varphi)\}.\]
	Accordingly, the {\em maximal spectrum\/} of $\varphi$ is set as
	\[\{\lambda\in\sigma(\varphi):\, |\lambda|=r(\varphi)\text{ and
	}d_{\varphi,\lambda}=d_\varphi\}.\]
\end{defn}
Now from equation \eqref{powers}, since the binomial coefficients $
\binom{n}{j}$ are increasing for $0\leq j\leq [n/2],$ for large $n$,
$ \|[\varphi^n]_\beta\|_\infty $ occurs at the top right hand corner
of the powers of Jordan blocks with maximal degeneracy index of
maximal modulus spectral values. Then, it is evident that for any
Jordan basis $\beta$ of $\varphi$, we have

\begin{equation}\label{outer:norm}
	\|\varphi^n\|_{\beta}=
	\| [\varphi^n]_\beta\|_{\infty}=\binom{n}{d_\varphi-1}r(\varphi
	)^{n-d_{\varphi }+1},
\end{equation}
for all but finitely many $n$. To demonstrate
that for a positive  map $\varphi$ on $\mathcal{A}$ with
$r(\varphi)>0$, $\widehat{\varphi}$ is a well defined non-zero
positive  map, we need the following lemma.

\begin{lem}\label{outer:lem5}
	Let $\{V_n\}$ be a sequence of positive $m\times m$ matrices such
	that
	\[\lim\limits_{N\to\infty}\frac{1}{N}\sum\limits_{n=1}^N V_n=0.\]
	Then there exists a sub-sequence $\{V_{n_k}\}$ converging to $0$.
\end{lem}
\begin{proof}
	Let $W_N=\frac{1}{N}\sum\limits_{n=1}^N V_n$. Clearly, $\{W_N\}$ is
	a sequence of positive matrices converging to $0$, in particular
	$\text{Tr}(W_N)\to 0$. Let $\liminf\, \text{Tr}(V_n)=t$. Given any
	$\epsilon>0$, there is $k\in\mathbb{N}$ such that $\text{Tr}(V_n)>
	t-\epsilon$ for all $n>k$. Then for $N>k$,
	$$\text{Tr}(W_N)=\frac{1}{N}\sum\limits_{n=1}^k \text{Tr}(V_n)+\frac{1}{N}\sum\limits_{n=k+1}^N \text{Tr}(V_n)>\frac{1}{N}
	\sum\limits_{n=1}^k \text{Tr}(V_n)+(t-\epsilon)(1-k/N)$$ Taking
	limit as $N\to\infty$, we obtain $(t-\epsilon)1\le0$ for all
	$\epsilon>0$. Therefore, $t=0$ and hence there is a sub-sequence
	$\{\text{Tr}(V_{n_k})\}$ converging to $0$. Since $V_{n_k}$'s are
	positive matrices, we conclude that $\{V_{n_k}\}$ converges to $0$.
\end{proof}

We now prove that the maximal part of $\varphi$ defines a non-zero
positive  map.
\begin{prop}\label{outer:prop3}
	Let $\varphi$ be a positive  map on a finite dimensional
	$C^*$-algebra $\mathcal{A}$ with $r(\varphi)>0$. Then
	$\widehat{\varphi}$ is a non-zero positive map on $\mathcal{A}$. If
	$\varphi $ is a CP map, then so is $\widehat{\varphi}.$
	
\end{prop}
\begin{proof}
	Recall from \eqref{outer:eq4} that
	$\widehat{\varphi}=\lim\limits_{N\to\infty}\frac{1}{N}\sum\limits_{n=1}^N\frac{\varphi^n}
	{\lVert\varphi^n\rVert_{\beta}}$ where $\beta$ is a Jordan basis of
	$\varphi$. To prove  the well-definedness of $\widehat{\varphi} $,
	that is, the convergence of the above limit, we first consider the
	case $r(\varphi)=1$.
	
	Now, since $\mathcal{A}$ is finite dimensional, it is enough to
	ensure  pointwise convergence on Jordan basis vectors. Let $x \in
	\beta$ be any Jordan basis vector. Then $x$ is a generalized
	eigenvector. Suppose it is of rank $k$  corresponding to eigenvalue
	$\lambda$, that is, $(\varphi-\lambda)^{k}x=0$ and $(\varphi
	-\lambda )^{k-1}x\neq 0$. Then, there is a Jordan chain $\{x_{k-1},
	x_{k-2}, \ldots, x_0\}\subseteq \beta$ where
	$x_j=(\varphi-\lambda)^jx$ for all $0\leq j\leq (k-1) $.  Observe
	that, for any $n\in \mathbb{N}$,
	\begin{equation}\label{outer:eq11}
		\varphi^n(x)=\varphi^n(x_0)=\sum\limits_{j=0}^{k-1}
		\binom{n}{j}\lambda^{n-j}x_j
	\end{equation}
	and there is $N_0\in\mathbb{N}$ such that $\|\varphi^n\|_{\beta}=
	\binom{n}{d_{\varphi}-1}$ for all $n\ge N_0$. Note that for
	$|\lambda|=1$,  $d_{\varphi,\lambda}\leq d_{\varphi}$, by the
	definition of $d_{\varphi }.$  Thus cardinality of $\{x_0,\dots,
	x_{k-1}\}$ corresponding to $|\lambda|=1$ is at most $d_{\varphi}$.
	Hence, it suffices to look at the following three cases.
	
	\noindent \textbf{Case I.} First assume that $|\lambda|<1$. Then, for any $p\in\mathbb{N}$, we have
	\[\lim\limits_{n\to\infty}\frac{\binom{n}{p}}{\|\varphi^n\|_\beta}\lambda^n
	=\lim\limits_{n\to\infty}\frac{\binom{n}{p}}{\binom{n}{d_{\varphi}-1}}\lambda^n=0.\]
	Therefore, using \eqref{outer:eq11}, we infer that
	$$\lim\limits_{m\to\infty}\frac{1}{m}\sum\limits_{n=1}^m\frac{\varphi^n(x_0)}{\lVert
		\varphi^n\rVert_\beta}=\sum\limits_{j=0}^{k-1}\lim\limits_{m\to\infty}
	\frac{1}{m}\sum\limits_{n=1}^m\frac{\binom{n}{j}}{\|\varphi^n\|_\beta}\lambda^{n-j}x_j=0
	.$$ Hence, when $|\lambda| <1$, $\widehat{\varphi}(x)=0$.
	
	\noindent \textbf{Case II.} Let $|\lambda|=1$ and $k < d_{\varphi}$. Therefore, by \eqref{outer:eq11},
	$$\lim\limits_{m\to\infty}\frac{1}{m}\sum\limits_{n=1}^m\frac{\varphi^n(x_0)}
	{\lVert
		\varphi^n\rVert_\beta}=\sum\limits_{j=0}^{k-1}\lim\limits_{m\to\infty}
	\frac{1}{m}\sum\limits_{n=1}^m\frac{\binom{n}{j}}{\lVert
		\varphi^n\rVert_\beta}\lambda^{n-j}x_j.$$
	For $0\leq j \leq k-1<d_\varphi -1$, we observe that
	\[\lim\limits_{n\to\infty}\left|\frac{\binom{n}{j}}{\lVert \varphi^n\rVert_\beta}\lambda^{n-j}\right|
	=\lim\limits_{n\to\infty}\frac{\binom{n}{j}}{\binom{n}{d_{\varphi}-1}}=
	0 .\]
	Hence, we have
	\[\lim\limits_{m\to\infty} \frac{1}{m}\sum\limits_{n=1}^m
	\frac{\binom{n}{j}}{\binom{n}{d_{\varphi}-1}}\lambda^{n-j}x_j= 0,\]
	that is, $\widehat{\varphi}(x)=0$.
	
	\noindent\textbf{Case III.} Now let $|\lambda|=1$ and
	$k=d_{\varphi}$. Then for $0\leq j < d_{\varphi}-1$, we similarly
	have
	\[
	\lim\limits_{m\to\infty}
	\frac{1}{m}\sum\limits_{n=1}^m\frac{\binom{n}{j}}{\|\varphi^n\|_\beta}\lambda^{n-j}x_j=
	0.\] Therefore,
	\begin{equation}\label{outer:eq12}
		\widehat{\varphi}(x)=\sum\limits_{j=0}^{d_{\varphi}-1}\lim\limits_{m\to\infty}
		\frac{1}{m}\sum\limits_{n=1}^m\frac{\binom{n}{j}}{\|\varphi^n\|_\beta}\lambda^{n-j}x_j=
		\lim\limits_{m\to\infty}\frac{1}{m}\sum\limits_{n=1}^m\frac{\binom{n}{d_{\varphi}-1}}{\|\varphi^n\|_\beta}
		\lambda^{n-d_{\varphi}+1}x_{d_{\varphi }-1}.
	\end{equation}
	Now, for $\lambda \neq 1 $,
	$$\lim\limits_{m\to\infty} \frac{1}{m}\sum\limits_{n=1}^m\frac{\binom{n}{d_{\varphi}-1}}{\lVert \varphi^n\rVert_\beta}\lambda^{n-d_{\varphi}+1}=\lim\limits_{m\to\infty} \frac{1}{m}\sum\limits_{n=N_0}^m\frac{\binom{n}{d_{\varphi}-1}}{\|\varphi^n\|_\beta}\lambda^{n-d_{\varphi}+1}=\lambda^{1-d_{\varphi}}\lim\limits_{m\to\infty}\frac{1}{m}\sum\limits_{n=N_0}^m \lambda^n=0.$$
	Thus, for $|\lambda|=1$, $k=d_{\varphi}$ and $\lambda \neq 1$ we have, $\widehat{\varphi}(x)=0$.
	
	\noindent For $\lambda=1$, by \eqref{outer:eq12},
	\[ \widehat{\varphi}(x)= \lim\limits_{m\to\infty}\frac{1}{m}\sum\limits_{n=N_0}^m
	1^{n-d_{\varphi}+1}x_{d_\varphi -1}=x_{d_\varphi -1}. \] Therefore,
	we have shown the convergence of the limit at all points of $\beta$.
	This proves the existence of $\widehat{\varphi}$ when $r(\varphi
	)=1.$
	
	For general positive map $\varphi$ with $r=r(\varphi)>0$, let
	$\psi=\varphi/r$ and $\beta'$ be a Jordan basis of $\psi$.
	Then $\psi$ is a non-zero positive map on $\mathcal{A}$ of spectral
	radius $1$. From the above discussion, $\widehat{\psi}$ exists. We
	claim that $\varphi$ and $\psi$ have same maximal degeneracy index.
	From the minimal polynomials of $\varphi$ and $\psi$, we observe
	that $d_{\varphi,\lambda}=d_{\psi,\lambda/r}$ for all
	$\lambda\in\sigma(\varphi)$. Hence, we have
	\[d_{\varphi}=\max\limits_{\lambda\in\sigma(\varphi)}\{d_{\varphi,\lambda}:\, |\lambda|=r\}=\max\limits_{\lambda\in\sigma(\varphi)}\{d_{\psi,\lambda/r}:\, |\lambda|=r\}=\max\limits_{\mu\in\sigma(\psi)}\{d_{\psi,\mu}:\, |\mu|=1\}=d_\psi.\]
	For simplicity, we denote $d=d_\varphi=d_\psi$. Recall from
	\eqref{outer:norm} that there is $N_1\in\mathbb{N}$ such that
	$\|\psi^n\|_{\beta'}=\binom{n}{d-1}$ and
	$\|\varphi^n\|_{\beta}=\binom{n}{d-1}r^{n-d+1}$ for all $n\ge
	N_1$. Now, it is evident that
	\begin{equation}\label{outer:MaximalPartRelation}
		\widehat{\left(\frac{\varphi}{r}\right)}=\widehat{\psi}=\lim\limits_{N\to\infty}\frac{1}{N}\sum\limits_{n=1}^N\frac{\psi^n}{\|\psi^n\|_{\beta'}}=\lim\limits_{N\to\infty}\frac{1}{N}\sum\limits_{n=N_1}^N\frac{\varphi^n}{\binom{n}{d-1}r^n}=\frac{\widehat{\varphi}}{r^{d-1}}.
	\end{equation}
	This completes the proof of existence of $\widehat{\varphi}$ for all
	positive map $\varphi$ with $r(\varphi)>0$. Being a limit of positive maps, $\widehat{\varphi}$ is positive. If $\widehat{\varphi}=0$, then
	\[\widehat{\varphi}(1)=\lim\limits_{N\to\infty}\frac{1}{N}\sum\limits_{n=1}^N\frac{\varphi^n(1)}{\lVert\varphi^n\rVert}_{\beta}=0.\]
	By  Lemma \ref{outer:lem5}, there exists a subsequence
	$\frac{\varphi^{n_k}(1)}{\lVert\varphi^{n_k}\rVert}_{\beta}$
	converging to $0$ in norm of $\mathcal{A}$. But since
	$\varphi^{n_k}$ is a positive map and $\lVert\cdot\rVert_{\beta}$ is
	equivalent to $\lVert\cdot\rVert$, there is some $c>0$ such that
	$$\lVert\varphi^{n_k}(1)\rVert=\lVert\varphi^{n_k}\rVert\ge c\lVert\varphi^{n_k}\rVert_{\beta}.$$
	This contradicts the fact that
	$\frac{\varphi^{n_k}(1)}{\lVert\varphi^{n_k}\rVert_{\beta}}\to0$.
	Therefore, $\widehat{\varphi}$ is a well defined, non-zero, positive
	map as needed. From the definition of $\widehat{\varphi }$, it is clear
	that if $\varphi $ is CP, so is $\widehat{\varphi } .$
\end{proof}
\begin{rem}\label{outer:note1}
	From the proof above, it is clear that, if $x$ is a generalized
	eigenvector of $\varphi $ corresponding to the eigenvalue $r(\varphi)$, then
	\begin{equation} \widehat{\varphi}(x)= (\varphi-r(\varphi))^{d_\varphi-1}x.
	\end{equation} On the other hand, if $x$ is any other generalized eigenvector, then
	$\widehat{\varphi } (x)=0.$
	
\end{rem}

\begin{thm}\label{outer:thm7}
	Let $\varphi$ be a non-zero positive map on a finite dimensional
	$C^*$-algebra. Then the spectral radius  $r(\varphi)$ belongs to the
	maximal spectrum of $\varphi$.
\end{thm}
\begin{proof}
	If $r(\varphi ) =0$, there is nothing to prove. Observe from Remark
	\ref{outer:note1} that  $\widehat{\varphi}$ is non-zero only when
	$r(\varphi)$ is an element of the maximal spectrum of $\varphi$. Since
	$\widehat{\varphi}$ is non-zero by Proposition \ref{outer:prop3},
	$r(\varphi)$ belongs to the maximal spectrum of $\varphi$.
\end{proof}

Now, we explore some useful properties of the maximal part of
$\varphi$ in the following theorem. It tells us as to how to
explicitly obtain an eigenvector with spectral radius as the
eigenvalue.
\begin{thm}\label{outer:thm6}
	Let $\varphi$ be a positive map on a  finite dimensional
	$C^*$-algebra $\mathcal{A}$ with $r(\varphi)>0$. Then the maximal
	part $\widehat{\varphi}$ is a non-zero positive map  satisfying the
	following properties:
	\begin{enumerate}
		\item $\varphi\widehat{\varphi}=\widehat{\varphi}\varphi=r(\varphi )\widehat{\varphi}$.
		\item $\widehat{\varphi}(1)$ is a positive eigenvector of $\varphi$ corresponding to the eigenvalue $r(\varphi ).$
		\item There are two possibilities with $\widehat{\varphi}$ which are:
		$$\widehat{\varphi}^2=\widehat{\varphi}\iff d_\varphi=1$$
		and
		\[\widehat{\varphi}^2=0\iff
		d_\varphi>1.\]
		\item If $\widehat{\varphi}^2=\widehat{\varphi}$, then for any element $x$,
		$\varphi(x)=r(\varphi )x$ if and only if $\widehat{\varphi }(x)=x.$
	\end{enumerate}
\end{thm}
\begin{proof}
	For notational simplicity, we take $d=
	d_{\varphi }$.
	\begin{enumerate}
		\item  Recall from the equation \eqref{outer:MaximalPartRelation}  that $\widehat{(\varphi/r)}=\widehat{\varphi}/r^{d-1}$.
		Hence, without loss of generality, we assume $r(\varphi)=1$ and show
		that
		$\varphi\widehat{\varphi}=\widehat{\varphi}\varphi=\widehat{\varphi}$.
		From Theorem \ref{outer:thm7}, we know that $1$ belongs to the
		maximal spectrum of $\varphi$.   Since $\beta$ is a Jordan basis of $\varphi$,
		$[\varphi]_\beta$ is of the form
		\[[\varphi]_\beta=  [\oplus _k(I+N_d) ]\oplus A\]
		where $N_d$ is the $d\times d$ Jordan nilpotent matrix of nilpotency
		$d$ so that $(I+N_d) $ is the largest Jordan block of $\varphi$
		corresponding to eigenvalue $1$, and $k$ indicates its multiplicity
		in the Jordan form of $\varphi$, and $A$ consists of the remaining
		Jordan blocks. Then, from Remark \ref{outer:note1}, we get that
		\[[\widehat{\varphi}]_\beta= \left[\oplus _kN_d^{d-1}\right]\oplus 0.\]
		Since $(I+N_d)N_d^{d-1}=N_d^{d-1}(I+N_d)=N_d^{d-1}$, we have
		$\varphi\widehat{\varphi}=\widehat{\varphi}\varphi=\widehat{\varphi}$.
		
		\item Since $\widehat{\varphi}$ is a non-zero
		positive map, $\widehat{\varphi}(1)$ is non-zero, positive element
		of $\mathcal{A}$. From $(1)$, we obtain
		$\varphi(\widehat{\varphi}(1))=r(\varphi )\widehat{\varphi}(1)$.
		\item We observe that
		\[\widehat{\varphi}^2=\widehat{\varphi}\iff N_{d}^{{d}-1}N_{d}^{{d}-1}=N_{d}^{{d}-1}\iff {d}=1 \text{ and}\]
		
		\[\widehat{\varphi}^2=0\iff N_{d}^{{d}-1}N_{d}^{{d}-1}=0\iff {d}>1.\]
		\item If $\widehat{\varphi}(x)=x,$ then
		\[\varphi(x)=\varphi\widehat{\varphi}(x)=r\,\widehat{\varphi}(x)=r\,x\]
		If $\varphi (x) =rx$, then $\varphi^n(x)=r^n\,x$ for all $n$. Since
		$\widehat{\varphi}^2=\widehat{\varphi}$, it follows from (3) that
		$d=1$. Thus,
		$\lVert\varphi^n\rVert_\beta=\binom{n}{d_\varphi-1}r^{n-d+1}=r^n$ for all but finitely many $n$'s.
		Therefore,
		\[\widehat{\varphi}(x)=\lim\limits_{N\to\infty}\frac{1}{N}\sum\limits_{n=1}^N\frac{\varphi^n(x)}
		{\lVert\varphi^n\rVert_\beta}=x.\]
	\end{enumerate}
\end{proof}

Theorem \ref{outer:thm7} and part (2) of Theorem \ref{outer:thm6} provide an alternate treatment to the Perron-Frobenius theorem.

\begin{rem}\label{outer:rem1}
	If $\varphi$ is an irreducible positive map on a finite dimensional
	$C^*$-algebra, then Evan and H\o egh-Krohn (\cite{EvaDE:HoeR78}) has proved the equality $r(\varphi)=\sup\limits_{x\ge0}\inf\{s>0:\,\varphi(x)\le s\,x \}$, which in particular implies that $r(\varphi)>0$. Therefore, we can define its maximal part. Theorem \ref{outer:thm10} asserts that the
	eigenspace corresponding to eigenvalue  $r(\varphi )$ is of
	dimension $1$ and is spanned by a strictly positive element
	$L\in\mathcal{A}$. Since by Theorem \ref{outer:thm6},
	$\widehat{\varphi}(1)$ is a positive eigenvector of $\varphi$
	corresponding to $r$, without loss of generality, we can assume that
	$L=\widehat{\varphi}(1)$. Also, by Theorem
	\ref{outer:thm6}, it is evident that $\widehat{\varphi}(X)$ belongs to the eigenspace of $\varphi$ corresponding to $r(\varphi)$ for all
	$X\in\mathcal{A}$. Therefore, there exists a state $\psi$ on
	$\mathcal{A}$ such that $\widehat{\varphi}(X)=\psi(X)\, L$ for all
	$X\in\mathcal{A}$, with  $\varphi(L)=r\,L$.
\end{rem}

\section{Algebra generated by Choi-Kraus coefficients}\label{outer:sec:algebra}

Given a $p$-tuple $(A_1, A_2, \ldots , A_p)$ of $m\times m$ matrices, computing a basis for the algebra generated by them in $M_m$ is an important problem in linear algebra with various practical applications. Throughout this section, the ``algebra generated by'' a tuple refers to the associative algebra over $\mathbb{C}$, that is, the complex linear span of all finite words formed by these matrices. We emphasize that this algebra is not generated in the $C^*$-sense: it need not be self-adjoint and may be non-unital. In contrast, the completely positive maps act on the full matrix algebra $M_m$, which is naturally a unital $C^*$-algebra. In \cite{PasJE19}, Pascoe shows
that this basis computation can be done by looking at the vectorization of a certain
matrix. The $p$-tuple can be considered as the Choi-Kraus
coefficients of the CP map $\tau $ on $M_m$ defined by
\[\tau (X) = \sum _{i=1}^pA_i^*XA_i, \quad X\in M_m.\]
Recasting the work of Pascoe and going further, we observe that the
elements of the algebra generated by these coefficients can be
obtained as Choi-Kraus coefficients of certain other CP maps
constructed out of $\tau $. There are many ways we can do this. We also look at irreducibility of completely positive maps, in light of the study of the maximal part done in the previous section.

As observed in Proposition \ref{outer:irreducible}, the irreducibility
property of a CP map on
$M_m$ is characterised in terms of the algebra
generated by its Choi coefficients.  For this purpose, we want to study the algebra
generated by some family of matrices in terms of the coefficient
space and Choi rank of some associated CP maps.

\begin{thm}\label{outer:thm12}
	Let $\mathcal{S}=\{A_1,\ldots,A_p\}$ be a $p$-tuple  of $m\times m$
	matrices for some $p\geq 1$ and let  $\langle\mathcal{S}\rangle _0$
	be the unital sub-algebra of $M_m$ generated by $A_i$'s and let
	$\langle\mathcal{S}\rangle _1$ be the (not-necessarily unital)
	algebra generated by $A_i$'s. Let $\tau (X)=\sum_{i=1}^p A_i^*
	XA_i$ be the CP map associated with $\mathcal{S}$. Let $\{a_n:
	n\geq 0\}$ be a sequence of strictly positive scalars such that
	$\sum _{n=0}^\infty a_n\tau ^n$ converges in norm.  Take
	$$\sigma _0= \sum _{n=0}^\infty a_n\tau ^n,~~\sigma _1= \sum
	_{n=1}^\infty a_n\tau ^n.$$ Then for $j=0,1$, the coefficient
	space of $\sigma _j$ is equal to $\langle\mathcal{S}\rangle _j$.
	Consequently, the dimension of $\langle\mathcal{S}\rangle _j$ is
	equal to the Choi rank of $\sigma _j$ and $A\in
	\langle\mathcal{S}\rangle _j$ iff there exists a scalar $q>0$  such
	that the map $\sigma _j- q\alpha _A$ is CP.

\end{thm}

\begin{proof}
	Note that $\tau^0(X)=X$ and for $n\geq 1$,
	$$\tau^n(X)=\sum\limits_{i_1,\ldots,i_n}(A_{i_1}A_{i_2}\cdots
	A_{i_n})^*X(A_{i_1}A_{i_2}\cdots A_{i_n}).$$ Therefore, we have
	\begin{equation}\label{outer:eq5}
		\sigma_0(X)=\sum\limits_{n=0}^{\infty
		}a_n{\sum\limits_{i_1,\ldots,i_n}}(A_{i_1}A_{i_2}\cdots
		A_{i_n})^*X(A_{i_1}A_{i_2}\cdots A_{i_n}).
	\end{equation}
	Let $\mathcal{M}_{0}$ be the coefficient space of $\sigma_0$.
	Then from  expression \eqref{outer:eq5}, it follows that
	$\mathcal{M}_{0}=\text{span}\{I, A_{i_1}A_{i_2}\cdots A_{i_n}:\,
	i_1,\ldots,i_n\in\{1,\ldots, p\}, n\geq 1 \}$ which is nothing but
	the algebra $\langle\mathcal{S}\rangle_{0}$. Therefore, we can
	conclude that the unital sub-algebra of $M_m$ generated by $A_i$'s
	is same as the coefficient space of $\sigma_ 0$. Similar
	argument works for $\sigma _1$. The remaining statements are clear
	from Proposition \ref{outer:prop2}.
\end{proof}

This theorem provides an abundant supply of CP maps which can be
used to determine the algebras generated by Choi-Kraus coefficients
of CP maps.  We single out two particularly useful examples. In a
matrix form, the CP map $\gamma _0$ was considered by Pascoe
\cite{PasJE19}.

\begin{cor}\label{outer:cor}
	Consider the set up of Theorem \ref{outer:thm12}.  Suppose $b>0$ is
	a scalar such that $br(\tau )<1$ and let $d>0$. Then
	$$\gamma _0:= (1-b\tau )^{-1}, ~~ \eta _0:=\exp (d\tau) $$
	are CP maps with their coefficient spaces equal to the unital
	algebra $\langle\mathcal{S}\rangle _0$ and
	$$\gamma _1= \gamma _0-1, ~~ \eta _1= \eta _0-1$$
	are CP maps with their coefficient spaces equal to the algebra
	$\langle\mathcal{S}\rangle _1$.
\end{cor}

\begin{proof}
	This is clear from Theorem \ref{outer:thm12}, by considering power
	series expansions:
	$$\gamma _0= \sum _{n=0}^\infty b^n\tau ^n, ~~\eta _0= \sum
	_{n=0}^\infty \frac{d^n}{n!}\tau ^n.$$ Note that the power series
	for $\gamma _0$ converges as $r(b\tau)<1.$
\end{proof}

\begin{rem}
	Theorem \ref{outer:thm12} and Corollary \ref{outer:cor} are stated
	for CP maps in finite dimensions. However, it is clear that they can
	be extended to normal completely positive maps on
	$\mathscr{B}(\mathcal{H})$, with separable infinite dimensional
	Hilbert space $\mathcal{H}$, where the power series under
	consideration are assumed to be convergent in strong operator
	topology.
\end{rem}

\begin{rem} If one is interested in analyzing the unital or
	non-unital $*$-algebra generated by a tuple $A_1,\ldots , A_p$ of
	matrices, one can consider the CP map $X\mapsto \sum
	_{i=1}^p(A_i^*XA_i+A_iXA_i^*)$ and the methods of Theorem
	\ref{outer:thm12} are applicable.
\end{rem}

Here in Theorem \ref{outer:thm12},  we use power series with
infinitely many terms and examples in Corollary \ref{outer:cor}
could be useful due to computational advantages. However, due to
dimension considerations, it is always enough to consider finitely
many terms. This can be seen as follows.

Once again, consider the set up of Theorem \ref{outer:thm12}. For
$k\geq 0$, let $\mathcal{M}_k$ be the coefficient space of $\sum_{n=0}^ka_n\tau ^n.$ Then, it is clear that $\mathcal{M}_0\subseteq
\mathcal{M}_1\subseteq \mathcal{M}_2\subseteq \cdots $ and all of
them are subspaces of $M_m$. We claim that this increasing sequence
stabilizes after a few steps.  Suppose
$\mathcal{M}_k=\mathcal{M}_{k+1}$ for some $k$. Then since any
$(k+1)$-length product $A_{j_1}A_{j_2}\cdots A_{j_{k+1}}$ is in the
coefficient space $\mathcal{M}_{k+1}$, we get
$$A_{j_1}A_{j_2}\cdots A_{j_{k+1}}\in
\mbox{span}\{ A_{i_1}A_{i_2}\cdots A_{i_r}: 0\leq r\leq k, 1\leq
i_1, \ldots i_r\leq p\}.$$
Then, it follows that any $(k+2)$-length
product $A_{j_1}A_{j_2}\cdots A_{j_{k+2}}$ is a linear combination
of $(k+1)$-length products of $A_j$'s. In other words,
$\mathcal{M}_{k+2}= \mathcal{M}_{k+1}.$ By induction,
$\mathcal{M}_k=\mathcal{M}_l$ for all $k\leq l.$ If
$\mathcal{M}_j\subsetneqq \mathcal{M}_{j+1}$, then their dimensions
have to differ by at least $1$ and all of them are subspaces of $M_m$,
which has dimension $m^2$, implies that the stabilization has to
take place by index $m^2$. Consequently, we have proved the
following theorem.

\begin{thm}\label{outer:thm13}
	Consider the set up of Theorem \ref{outer:thm12}. Then for $j=0,1$
	there exists natural numbers $N_j \leq m^2$ such that the coefficient space of $\sum _{n=j}^{N_j}a_n\tau ^n$ is equal to
	$\langle\mathcal{S}\rangle _j$ and the Choi rank of $\sum
	_{n=j}^{N_j}a_n\tau ^n$ is equal to the dimension of
	$\langle\mathcal{S}\rangle _j$.
\end{thm}

To check whether a CP map $\tau$ on $M_m$ is irreducible, one can construct a convenient CP map $\sigma_1$ as described in Theorem \ref{outer:thm12}. The combination of Theorem \ref{outer:thm12} and Proposition \ref{outer:irreducible} ensures that $\tau$ is irreducible if and only if $\sigma_1$ has a Choi rank of $m^2$.

Now, we proceed with our desired characterisation for irreducible CP
maps on  general finite dimensional $C^*$-algebras and a description
of their maximal parts.  Compared to maps on full matrix algebra (see Proposition \ref{outer:irreducible}), the situation is a bit more complicated
now. Consider a CP map $\tau $ with a specified Choi-Kraus type
decomposition. Whenever these coefficients do not admit a proper common invariant
subspace, $\tau$ is irreducible. However, as the following example illustrates, this condition is not always necessary.

\begin{eg}\label{outer:eg3}
	Let {$\mathcal{A}= M_1\oplus M_1$ be the algebra of $2\times 2$ complex diagonal matrices}, and let $\tau$ be the CP map
	defined by $\tau(X)=2\,\text{trace}(X)\,I$ on $\mathcal{A}$. We can express
	\[\tau(X)=A^*XA +B^*XB\text{ for all }X\in\mathcal{A},\]
	where $A=\begin{bmatrix}
		1 & 1 \\
		1 & 1
	\end{bmatrix}$ and $B=\begin{bmatrix}
		1 & -1\\
		1 & -1
	\end{bmatrix}$.
	Observe that $\tau$ is strictly positive, and therefore, it is irreducible. However, it exhibits a Choi-Kraus type expression involving $A$ and $B$ which share the non-trivial common invariant subspace $\text{span}\left\{\begin{pmatrix}1\\1\end{pmatrix}\right\}$. This shows that even if $A$ and $B$ share a common non-trivial invariant subspace, $\tau$ can still be irreducible when the domain is not a full matrix algebra. 
	
	Note that because $\mathcal{A} = M_1 \oplus M_1$ is commutative, a vector in $\mathbb{C}^2$ with positive entries is algebraically equivalent to a positive element in $\mathcal{A}$. For instance, the identity matrix $I = \begin{bmatrix} 1 & 0 \\ 0 & 1 \end{bmatrix}$ is an eigenvector of $\tau$ that is both a positive element in the $C^*$-algebra $\mathcal{A}$ and a matrix with non-negative entries. In a general non-commutative $C^*$-algebra, however, the concept of ``matrices with non-negative entries" is not applicable, making it essential to conceptually distinguish the underlying Hilbert space from the algebra itself.
\end{eg}

Consider a finite dimensional $C^*$-algebra $\mathcal{A}.$ Without
loss of generality,  $\mathcal{A}= M_{n_1}\oplus \dots \oplus
M_{n_d}$ for some natural numbers $n_1, n_2, \ldots , n_d$. Clearly,
$\mathcal{A}$ has a natural embedding in $M_m$ where
$m=n_1+\cdots+n_d$. Throughout, we will consider this embedding
of $\mathcal {A}$.  Now, we can canonically extend any CP map on
$\mathcal{A}$ to a CP map on $M_m$ as follows.
\begin{defn}\label{outer:def1}
	For a CP map $\tau$ on a finite dimensional $C^*$-algebra
	$\mathcal{A}=M_{n_1}\oplus\cdots\oplus M_{n_d}$, if $\iota:\mathcal{A}\hookrightarrow M_m$ is the canonical
	embedding of $\mathcal{A}$ into $M_m$ where $m=n_1+\dots+
	n_d$, $P_k$
	is the orthogonal projection of $M_m$ onto the $k^{\text{th}}$ diagonal block $M_{n_k}$ for each $k\in\{1,\dots, d\}$ and
	$\phi:M_m\to\mathcal{A}$ is defined by $\phi(X)=\sum_{k=1}^d
	P_k^* X P_k$, then the {\em canonical extension\/} of $\tau$ to $M_m$,
	denoted by $\mathbf{\widetilde{\tau}}$, is defined to be
	$\mathbf{\widetilde{\tau}=\iota\circ\tau\circ\phi}$.
\end{defn}
It is natural to expect a connection between the irreducibility of
$\tau$ and $\widetilde{\tau}$. To establish this, we need the
following lemma. Recall that a positive map is called strictly
positive if it sends all non-zero positive elements to strictly
positive elements.  The lemma tells us that for a positive matrix
having a block matrix decomposition, the diagonals can be perturbed
suitably to get a strictly positive matrix.
\begin{lem} \label{outer:lem11}
	Let $\{\psi_{ij}:M_{n_j}\to M_{n_i}|\, i,j\in\{1,\ldots,d\}\}$ be a collection of positive maps with $1+\psi_{ii}>0$ for all $i$ and $\psi_{ij}>0$ for all $i\ne j$ where $n_1,\ldots,n_d\in\mathbb{N}$. Let $X=(X_{ij})\in M_m$ be a non-zero positive (block) matrix where $X_{ij}\in M_{n_i\times n_j}$ and $m=n_1+\ldots+n_d$. Then the matrix $Y=(Y_{ij})\in M_m$ defined as
	\[Y_{ij}=
	\begin{cases}
		X_{ii}+\sum\limits_{k=1}^d\psi_{ik}(X_{kk}), & \text{ if }i=j\\
		X_{ij}, & \text{ if } i\ne j
	\end{cases}\]
	is strictly positive.
\end{lem}

\begin{proof}
	Observe that, we can write $Y=X+Z$ where $Z$ is a block diagonal
	positive matrix with diagonal entries $Z_{ii}=\sum_{k=1}^d
	\psi_{ik}(X_{kk})$. If all $Z_{ii}$'s are strictly positive, then
	$Z$, hence $Y$, becomes strictly positive matrix. If some $Z_{ll}$
	is not strictly positive, then $X_{kk}=0$ for all $k\ne l$. Then $X$
	turns into a block diagonal matrix with only non-zero entry
	$X_{ll}$. Then we have
	\[
	Y_{ij}=
	\begin{cases}
		\psi_{il}(X_{ll}), & \text{ if }i=j\ne l\\
		(1+\psi_{ll})(X_{ll}), & \text{ if }i=j=l\\
		0, & \text{ if }i\ne j
	\end{cases}.
	\]
	Thus, $Y$ becomes a block diagonal matrix with strictly positive diagonal entries. Therefore, $Y$ is a strictly positive matrix in either cases.
\end{proof}
\begin{prop}\label{outer:prop8}
	Let $\tau$ be a CP map on finite dimensional $C^*$-algebra
	$\mathcal{A}=M_{n_1}\oplus\cdots\oplus M_{n_d}$  and let
	$\widetilde{\tau}$ be its canonical extension to $M_m$ where
	$m=n_1+\cdots+n_d$. Then $\tau$ is irreducible if and only if
	$\widetilde{\tau}$ is irreducible. The spectral radii of $\tau $ and
	$\widetilde{\tau }$ are equal.
\end{prop}
\begin{proof}
	First assume that, $\widetilde{\tau}$ is irreducible. Then by Lemma
	2.1 of \cite{EvaDE:HoeR78},  $(1+\widetilde{\tau})^{m-1}$ is
	strictly positive. For any non-zero positive $X\in\mathcal{A}$, we
	infer that
	$(1+\tau)^{m-1}(X)=(1+\tau)^{m-1}(\phi(X))=(1+\widetilde{\tau})^{m-1}(\phi(X))>0$,
	proving irreducibility of $\tau$.
	
	Conversely, assume that $\tau$ is irreducible. Since $\tau$ is a
	positive map, for any $A=A_1\oplus \dots \oplus A_d\in \mathcal{A}$,
	we have $(1+\tau)^{m-1}(A)-A \geq 0$ and its $i^\text{th}$ diagonal
	block entry then must look like sum of the positive maps
	corresponding to each $j=1,\dots ,d$ that take $A_{j}\in M_{n_j}$ to
	some matrix in $M_{n_i}$. That is, there is a collection of positive
	maps $\{\psi_{ij}:M_{n_j}\to M_{n_i}|\, i,j\in\{1,\ldots,d\}\}$ such
	that
	\[((1+\tau)^{m-1}(A))_{ii}=A_{i}+\sum\limits_{j=1}^d\psi_{ij}(A_{j})\ \text{ for all }X\in\mathcal{A}.\]
	We claim that this collection of $\psi_{ij}$'s satisfy the hypothesis of Lemma \ref{outer:lem11}. Fix some $i\in \{ 1,\dots ,d \} $ arbitrarily and let $B\in M_{n_i}$ be any non-zero positive matrix. Define $A\in\mathcal{A}$ to be the block diagonal matrix whose $i^\text{th}$ diagonal block is $B$ and all other are zero. It follows from the irreducibility of $\tau$ that $(1+\tau)^{m-1}$ is strictly positive. Thus,
	we have $((1+\tau)^{m-1}(A))_{kk}>0$ for each $k$. Hence, in
	particular  \[
	((1+\tau)^{m-1}(A))_{ii} = (1+\psi_{ii})(B)>0
	\] and for $j\neq i$,
	\[
	((1+\tau)^{m-1}(A))_{jj} = \psi_{ji}(B)>0.
	\]
	As $B$ is arbitrary non-zero positive matrix in $M_{n_i}$, we
	get $1+\psi_{ii}>0$ and $\psi_{ij}>0$ for all $i\ne j$. Since $i,j$ are arbitrary, this proves the required claim. Now, let $X\in M_m$
	be non-zero positive. A simple computation shows that,
	\[ (1+\widetilde{\tau})^{m-1}(X)= (1+\iota \circ \tau\circ\phi)^{m-1}(X)=(X-\phi(X))+(1+\tau)^{m-1}(\phi(X)).\]
	Then,
	\[((1+\widetilde{\tau})^{m-1}(X))_{ij}=
	\begin{cases}
		X_{ii}+\sum\limits_{k=1}^d\psi_{ik}(X_{kk}), & \text{ if }i=j\\
		X_{ij}, & \text{ if } i\ne j
	\end{cases}.\]
	Hence, by Lemma \ref{outer:lem11}, we infer that
	$(1+\widetilde{\tau})^{m-1}(X)>0$ for all non-zero positive $X\in
	M_m$. Therefore,
	$\widetilde{\tau}$ is irreducible. To see the second part, observe
	that $\widetilde{\tau}^n=\iota\circ\tau^n\circ\phi$ for all
	$n\in\mathbb{N}$ and
	\[r(\widetilde{\tau})=\lim\limits_{n\to\infty}\lVert\widetilde{\tau}^n({1})\rVert^{1/n}=
	\lim\limits_{n\to\infty}\lVert\tau^n({1})\rVert^{1/n}=r(\tau).\]
\end{proof}

The following theorem says that if we are looking at a Choi-Kraus
decomposition $\tau (X)=\sum_{i=1}^pA_i^*XA_i$ for all $X\in \mathcal{A}$,
which satisfies also $\widetilde{\tau}(X)= \sum_{i=1}^pA_i^*XA_i$ for all $X\in M_m$, then the irreducibility of $\tau $ can be read of from
$A_i$'s. Note that this was not the case in Example \ref{outer:eg3}.

\begin{thm}\label{outer:thm20}
	Let $\tau$ be a CP map on finite dimensional $C^*$-algebra
	$\mathcal{A}=M_{n_1}\oplus\cdots\oplus M_{n_d}$ and
	$\widetilde{\tau}$ be its canonical extension to $M_m$ where
	$m=n_1+\cdots+n_d$.  Then, the following statements are equivalent:
	\begin{enumerate}
		\item $\tau$ is irreducible;
		\item $\widetilde{\tau}$ is irreducible;
		\item There exists matrices $A_1,\ldots,A_p\in M_m$ which generates $M_m$ as algebra such that $\widetilde{\tau}(X)=\sum_{i=1}^p A_i^*XA_i$ for all $X\in M_m$. In particular, $\tau(X)=\sum_{i=1}^p A_i^*XA_i$ for all
		$X\in\mathcal{A}$;
		\item There exist matrices $A_1,\ldots,A_p\in M_m$ which generates $M_m$ as algebra such
		that $\tau(X)=\sum_{i=1}^p A_i^*XA_i$ for all
		$X\in\mathcal{A}$.
	\end{enumerate}
\end{thm}
\begin{proof}
	We have already established the equivalence between (1) and (2) in Proposition \ref{outer:prop8}, and the implication of (3) leading to (4) is self-evident. Now, we prove the remaining implications.
	
	\noindent (2) $\Rightarrow$ (3): Suppose $\widetilde{\tau}$ has
	Choi-Kraus expression $\widetilde{\tau}(X)=\sum_{i=1}^p
	A_i^*XA_i$ for all $X\in M_m$ where $\{A_1,\ldots,A_p\}\subseteq M_m$. Since $\tau$ is
	irreducible, by Proposition \ref{outer:irreducible}, the algebra generated by $A_i$'s is whole of $M_m$.  In particular,
	for $X=\oplus _{j=1}^dX_j\in\mathcal{A}$, we have
	\[\tau(X)=\iota\circ\tau\circ\phi(X)=\widetilde{\tau}(X)=\sum\limits_{i=1}^p A_i^*X A_i.\]
	
	\noindent {(4)} $\Rightarrow$ {(1)}: Let us assume on contrary that $\tau$ is reducible. Then, there is a non-trivial projection $P\in\mathcal{A}$ such that $\tau(P)\le\lambda\, P$ for some $\lambda>0$. We show that $\text{Ker}(P)$ is invariant under all $A_i$'s. For $h\in\text{Ker}(P)$, we have
	\[\|PA_ih\|^2=\langle A_i^*PA_ih,h\rangle\le\lambda\langle Ph,h\rangle=0, \text{for all }i.\]
	Therefore, $A_i(\text{Ker}(P))\subseteq \text{Ker}(P)$ for all
	$i$, hence, $A_i$'s possess a non-trivial common invariant subspace
	$\text{Ker}(P)$. This contradicts to the fact that
	$A_1,\ldots,A_p$ generate the algebra $M_m$.
\end{proof}

Let us now give two examples of CP maps on finite dimensional
$C^*$-algebras whose irreducibility can be verified using the
analysis done so far. For instance, one may use Corollary
\ref{outer:cor} for $\widetilde{\tau }.$

\begin{eg}\label{outer:eg2}
	Let $\mathcal{A}=M_2\oplus M_1$ and $\tau:\mathcal{A}\to\mathcal{A}$ be defined as
	\[\tau\left(\begin{bmatrix}
		a & b & 0\\
		c & d & 0\\
		0 & 0 & e
	\end{bmatrix}\right)=\begin{bmatrix}
		a+e & 0 & 0\\
		0 & a+e & 0\\
		0 & 0 & d
	\end{bmatrix}.\]
	Then $\tau $ is irreducible. \end{eg}

\begin{eg}\label{outer:eg4}
	Let $\mathcal{A}=M_1\oplus M_1\oplus M_1$ and
	$\tau:\mathcal{A}\to\mathcal{A}$ be defined as
	\[\tau\left(\begin{bmatrix}
		a & 0 & 0\\
		0 & b & 0\\
		0 & 0 & c
	\end{bmatrix}\right)=\begin{bmatrix}
		b & 0 & 0\\
		0 & a+c & 0\\
		0 & 0 & b
	\end{bmatrix}.\]
	Then $\tau $ is irreducible.
\end{eg}

In Section \ref{Outer:sec:positivemap}, we have defined and studied
the maximal part $\widehat{\varphi}$  for any positive map $\varphi$
with $r(\varphi)>0$, on finite dimensional $C^*$-algebras. For CP maps, there is more to explore.
We observe that one can obtain a Choi-Kraus decomposition for a CP
map $\tau$ in such a way that the Choi-Kraus coefficients of
$\widehat{\tau}$ span a non-zero ideal in the algebra generated by
the Choi-Kraus coefficients of $\tau$. To proceed, we require the following technical lemma.

\begin{lem}\label{outer:Sequence_of_CP_maps}
	Let $\{\tau_n\}$ be a sequence of CP maps on $M_m$ converging to a
	CP map $\tau$. Suppose that there is a subspace $\mathscr{V}$ of
	$M_m$ such that
	$\mathcal{M}_{\tau_n}\subseteq\mathcal{M}_{\tau_{n+1}}\subseteq\mathscr{V}$
	for all $n$. Then, $\mathcal{M}_{\tau}\subseteq\mathscr{V}$.
\end{lem}
\begin{proof}
	Let $C_{n}$ and $C$ be the conjugated Choi matrices of the CP maps $\tau_n$ and
	$\tau$ respectively. Since $\lim\tau_n=\tau$, we have $\lim
	C_{n}=C$. Now, let $P_n$ and $P$ be the $m^2\times m^2$ projection
	matrices with $\text{range}(P_n)=\text{range}(C_{n})$ and
	$\text{range}(P)=\text{range}(C)$. As
	$\mathcal{M}_{\tau_n}\subseteq\mathcal{M}_{\tau_{n+1}}\subseteq\mathscr{V}$,
	by Proposition \ref{outer:choi_matrix}, we have
	$\text{range}(P_n)\subseteq\text{range}(P_{n+1})\subseteq\text{vec}(\mathscr{V})$.
	Let $Q$ be the projection onto the subspace
	$\text{vec}(\mathscr{V})\subseteq\mathbb{C}^{m^2}$. Then, $P_n\le
	P_{n+1}\le Q$ for all $n$. For any $h\in\mathbb{C}^m$, there is
	$h'\in\mathbb{C}^m$ such that $Ph=C h'$. We infer that
	\[\|P_nPh-Ph\|=\|P_nC h'-Ch'\|\le\|P_n(C h'-C_{n}h')\|+\|C_{n}h'-C h'\|\le 2\|C_{n}h'-C h'\|.\]
	Taking limit as $n$ tends to infinity, this implies that $\lim _n
	P_nP=P$. Also, since $\{P_n\}$ is an increasing sequence of
	projections with $P_n\le Q$, there exists a projection $R\in
	M_{m^2}$ such that $P_n$ converges to $R$ and $R\le Q$. Then,
	$RP=\lim P_nP=P$. Consequently,
	$\text{range}(P)\subseteq\text{range}(R)$, and thus,
	$\text{range}(P)\subseteq \text{range}(R)\subseteq
	\text{range}(Q)=\text{vec}(\mathscr{V})$. Therefore, we have
	$\mathcal{M}_\tau=\text{vec}^{-1}(\text{range}(P))\subseteq
	\mathscr{V}$.
\end{proof}
\begin{rem}
	In the above lemma, the family of coefficient spaces $\mathcal{M}_{\tau_n}$ may not increase to $\mathcal{M}_\tau$, as demonstrated
	by the following simple example. Let $\tau_n$ be the CP map on $M_2$ defined as
	\[\tau_n(X)=\text{trace}(X)
	\begin{bmatrix}
		1 & 1-\frac{1}{n}\\
		1-\frac{1}{n} & 1
	\end{bmatrix}, ~~ \forall X\in M_2.\]
	It converges to the CP map $\tau(X)=\text{trace}(X)\left[\begin{matrix}
		1 & 1\\
		1 & 1
	\end{matrix}\right]$, $X\in M_2$.
	Observing that the rank of each Choi matrix $C_{\tau_n}$ is $4$, we have $\mathcal{M}_{\tau_n}=M_2$ for all $n$. However, the
	rank of the Choi matrix $C_{\tau}$ is $2$, indicating that
	$\mathcal{M}_\tau\subsetneqq M_2$.
\end{rem}

\begin{thm}\label{outer:thm8}
	Let $\tau$ be a CP map on a finite dimensional $C^*$-algebra
	$\mathcal{A}=M_{n_1}\oplus\cdots\oplus M_{n_d}$ with $r(\tau)>0$.
	Consider canonical extensions of $\tau , \widehat{\tau}$. Then
	\begin{equation}
		\widetilde{\widehat{\tau}}=\widehat{\widetilde{\tau}}.
	\end{equation}
	For any two  Choi-Kraus decompositions:
	\[\widetilde{{\tau}}(X)=\sum\limits_{i=1}^p A_i^*X A_i, ~~~~\widetilde{\widehat{\tau}}(X)=\sum\limits_{j=1}^{q}B_j^*XB_j, ~~
	X\in M_m\] of $\widetilde{\tau }, \widetilde{\widehat{\tau}}$
	respectively, the vector space $\text{span}\{B_1,\ldots, B_{q}\}$ is
	a subalgebra of $M_m$ and it is a non-zero ideal in the algebra
	generated by $A_1,\ldots,A_p$.
\end{thm}
\begin{proof}
	Observe that $\tau :\mathcal{A}\to \mathcal{A}$,
	$\widetilde{\tau}:M_m\to M_m$ and $\mathcal{A}$ is a subspace of
	$M_m$. With respect to the Hilbert-Schmidt inner product, we
	decompose $M_m$ as $M_m= \mathcal{A}\oplus \mathcal{A}^\perp $ and
	with respect to this decomposition, $\widetilde{\tau }= \tau \oplus
	0$.
	Now, it is clear that given a Jordan basis $\beta $ for $\tau $ (basis
	elements span $\mathcal{A}$), we can extend it to a Jordan basis
	$\beta '$ for $\widetilde{\tau }$ (basis elements span $M_m$) such
	that elements of $\beta '\setminus \beta $ form a basis of
	$\mathcal{A}^\perp.$ In particular, $\widehat{\tau }(X)=0$ for all
	$X\in \beta '\setminus \beta .$ Then it is easy to see that
	$\lVert\widetilde{\tau}^n\rVert_{\beta'}=\lVert\tau^n\rVert_{\beta}$
	for all $n\in\mathbb{N}$. From Proposition \ref{outer:prop8},
	$r(\widetilde{\tau })=r(\tau )>0$ and hence
	\[\widehat{\widetilde{\tau}}=\lim\limits_{N\to\infty}\frac{1}{N}\sum\limits_{n=1}^N\frac{\widetilde{\tau}^n}{\lVert\widetilde{\tau}^n\rVert_{\beta'}}=\lim\limits_{N\to\infty}\frac{1}{N}\sum\limits_{n=1}^N\frac{\iota\circ\tau^n\circ\phi}{\lVert\tau^n\rVert_{\beta}}=\iota\circ\widehat{\tau}\circ\phi = \widetilde{\widehat{\tau}}.\]
	This proves the first part.
	
	For the second part, let $\tau_N=\frac{1}{N}\sum\limits_{k=1}^N\frac{\widetilde{\tau}^k}{\lVert\widetilde{\tau}^k\rVert_{\beta'}}$ and $\mathscr{V}$ be the algebra generated by $A_i$'s. Then, $\lim\tau_N=\widehat{\widetilde{\tau}}$. Observe that $\mathcal{M}_{\widehat{\widetilde{\tau}}}=\text{span}\{B_j:1\le j\le q\}$ and $\mathcal{M}_{\tau_N}\subseteq\mathcal{M}_{\tau_{N+1}}\subseteq\mathscr{V}$ for all $N$. Therefore, by Lemma \ref{outer:Sequence_of_CP_maps}, it is evident that $\text{span}\{B_j:1\le j\le q\}$ is a subspace of the algebra generated by $A_i$'s. From  Theorem \ref{outer:thm6}, we know that either
	$\widehat{\widetilde{\tau}} ^2= \widehat{\widetilde{\tau}}$ or
	$\widehat{\widetilde{\tau }}^2= 0$. In the first case, the expression
	$$\sum _{j,k}B_k^*B_j^*XB_jB_k= \sum _jB_j^*XB_j, ~~X\in M_m$$
	represents two different Choi-Kraus decompositions of
	$\widehat{\widetilde{\tau }}$, implying $B_jB_k\in
	\text{span}\{B_j:1\leq j\leq q\}$. In the second case, $$\sum _{j,k}B_k^*B_j^*XB_jB_k=0, ~~X\in M_m$$ and hence $B_jB_k=0$ for all
	$j,k.$ Therefore, in either case $\text{span}\{B_j:1\leq j\leq q\}$
	is an algebra. Similarly, from Theorem \ref{outer:thm6}, we have
	$\widetilde{\tau}\,\widehat{\widetilde{\tau}}=\widehat{\widetilde{\tau}}\,\widetilde{\tau}=r(\tau)\,\widehat{\widetilde{\tau}}$.
	Thus,
	\[\sum\limits_{i,j}(A_i B_j)^*X(A_i B_j)=\sum\limits_{j,i}(B_jA_i)^*X(B_jA_i)=r(\tau)\,\sum\limits_{j=1}^{q}B_j^*X B_j\ \text{ for all } X\in M_m.\]
	Hence, by Proposition \ref{outer:prop2} each of $A_iB_j$ and
	$B_jA_i$ belong to the span of $B_j$'s, which implies
	$\text{span}\{B_j:1\le j\le q\}$ forms an ideal of the
	algebra generated by $A_i$'s. It is non-zero because
	$\widehat{\widetilde{\tau}}$ is a non-zero map.
\end{proof}

Now, we look at irreducible CP maps on finite dimensional
$C^*$-algebras. We need a technical lemma.

\begin{lem}\label{outer:lem6}
	Let $W_1, W_2$ be two linearly independent, positive operators on a
	Hilbert
	space $\mathcal{H}$. Then there exists a non-zero $W_3\in\text{span}\{W_1,W_2\}$ which is positive and singular.
\end{lem}
\begin{proof}
	If $W_1$ (or $W_2$) is singular, we may simply set $W_3 = W_1$ (respectively, $W_3 = W_2$), as it immediately satisfies the requirement. So, we may assume both $W_1$ and $W_2$ are non-singular. Let $d$ be the spectral radius of the positive operator $V={W_1}^{-1/2}W_2 {W_1}^{-1/2}$ and $W_3=W_1-d^{-1}W_2$ which is non-zero due to the linear independence of $W_1,W_2$. Since $\sigma(V)=\sigma({W_1}^{-1/2}W_2 {W_1}^{-1/2})=\sigma(W_1^{-1}W_2)$, we have $d\in\sigma(W_1^{-1}W_2)$. This implies $W_1^{-1}W_2-d\,I$ is singular which is equivalent to the singularity of $W_3=d^{-1}W_1(d\,I-W_1^{-1}W_2)$. We need to verify the positivity of $W_3$. For $x\in \mathcal{H}$, take $y=W_1^{1/2}x$. Then
	\[\langle x,(W_1-d^{-1}W_2) x\rangle=\langle y,y\rangle-d^{-1}\langle y,{W_1}^{-1/2}W_2 {W_1}^{-1/2}y\rangle=d^{-1}(d\langle y,y\rangle-\langle y,Vy\rangle)\]
	Since $V\le d\,I$, we infer that $\langle x,W_3 x\rangle\ge0$ for all $x\in\mathcal{H}$, proving positivity of $W_3$.
\end{proof}

We have observed in Remark \ref{outer:rem1} that if a positive map on a finite dimensional
$C^*$-algebra  is irreducible, then its spectral radius is strictly
positive. So, we can talk about its maximal part. Here, we connect the irreducibility of a CP map with that of its maximal part.

\begin{thm}\label{outer:thm9}
	Let $\tau$ be a CP map on a finite dimensional $C^*$-algebra
	$\mathcal{A}=M_{n_1}\oplus\cdots\oplus M_{n_d}$  and
	$m=n_1+\cdots+n_d$.  Then,  $\tau$ is irreducible if and only if
	$\widehat{\tau }$ is irreducible. In such a case, the following
	properties hold: \begin{enumerate}
		\item The maximal degeneracy index of $\tau $ is $1.$
		
		\item The map $\widehat{\tau}$ is
		a strictly positive map of rank $1$.
		
		\item  The CP map $\widetilde{\widehat{\tau}}$ has Choi rank $m^2$. In particular, there exists $m\times m$ matrices $\{B_1,\ldots,B_q\}$ which span $M_m$ and
		\[\widehat{\tau}(X)=\sum\limits_{j=1}^{q}B_j^*X B_j,~~ \forall X\in \mathcal{A}.\]
		
		\item  There is a faithful state $\psi$ on $\mathcal{A}$ and a strictly positive $L\in\mathcal{A}$ such that
		\[\widehat{\tau}(X)=\psi(X)\,L\ \text{ for all }X\in\mathcal{A}\] and
		(a)  $\tau(L)=r(\tau)\,L$, (b)  $\psi(L)=1$, and (c) $\psi\circ\tau=r(\tau)\,\psi$.
		
		\item There is no positive generalized eigenvector of $\tau$ corresponding to eigenvalue of modulus less than $r(\tau)$. In other words, the set $\left\{X\ge 0:\,\lim\frac{\tau^n(X)}{r(\tau)^n}=0\right\}$ only contains $0$.
	\end{enumerate}
\end{thm}
\begin{proof}
	The irreducibility of $\widehat{\tau}$ ensures the irreducibility of $\tau$ since any hereditary subspace invariant under $\tau$ remains invariant under $\tau^n$, and therefore under $\widehat{\tau}$ as well. Assume that $\tau$ is irreducible. Suppose that the CP maps $\widetilde{{\tau}}$ and $\widetilde{\widehat{\tau}}$ have the Choi-Kraus decompositions
	\[\widetilde{\tau}(X)=\sum\limits_{i=1}^p A_i^*X A_i,\ \widetilde{\widehat{\tau}}(X)=\sum\limits_{j=1}^{q}B_j^*X B_j\ \text{ for all } X\in M_m .\]
	According to Theorem \ref{outer:thm8}, $B_j$'s span a non-zero ideal in the algebra generated by the $A_i$'s. Since $\tau$ is irreducible, it follows from  Theorem \ref{outer:thm20} that the Choi-Kraus coefficients $A_i$'s of the canonical extension $\widetilde{\tau}$ generate the full matrix algebra $M_m$. Since $M_m$ is simple, we have $\text{span}\{B_j:1\le j\le q\}=M_m$. Therefore, $\mathcal{M}_{\widetilde{\widehat{\tau}}}=M_m$, in other words, $\widetilde{\widehat{\tau}}$ has Choi rank $m^2$. This proves part (3). By Proposition \ref{outer:prop2},
	we infer that $\alpha_I$ is dominated by $s\,\widetilde{\widehat{\tau}}$ for some $s>0$ where $\alpha_I(X)=X$ for all $X\in M_m$. This implies in particular that for all positive $X\in\mathcal{A}$, $\widehat{\tau}(X)\ge s^{-1}X$, thus confirming the faithfulness of $\widehat{\tau}$.
	
	Let $X\in\mathcal{A}$ be non-zero positive. Then, by Theorem \ref{outer:thm6}, $\tau(\widehat{\tau}(X))=r(\tau)\widehat{\tau}(X)$. For any $h\in\text{Ker}(\widehat{\tau}(X))$, we infer that
	\[\|\widehat{\tau}(X)^{1/2}A_ih\|^2\le\sum\limits_k \langle \widehat{\tau}(X)A_kh,A_kh\rangle=\langle\tau(\widehat{\tau}(X))h,h\rangle=r(\tau)\langle \widehat{\tau}(X)h,h\rangle=0\text{ for all }i.\]
	Therefore, $\widehat{\tau}(X)A_ih=0$ for all $h\in\text{Ker}(\widehat{\tau}(X))$, consequently implying $A_i(\text{Ker}(\widehat{\tau}(X)))\subseteq\text{Ker}(\widehat{\tau}(X))$ for all $i$. Since $A_i$'s generate the full matrix algebra $M_m$ and $\widehat{\tau}(X)$ is non-zero, $\text{Ker}(\widehat{\tau}(X))=0$, proving strict positivity of $\widehat{\tau}(X)$. Hence, $\widehat{\tau}$ is strictly positive, in particular, irreducible.
	
	Now, as per Theorem \ref{outer:thm6}, either $\widehat{\tau}^2=0$ or $\widehat{\tau}^2=\widehat{\tau}$. Since $\widehat{\tau}$ is strictly positive, we must have $\widehat{\tau}^2=\widehat{\tau}$, hence from Theorem \ref{outer:thm6}, $d_\tau=1$. This proves part (1).  Let us assume on contrary that $\widehat{\tau}$ has rank greater than $1$. Then, since $\mathcal{A}$ is span of its positive elements, there are positive elements $V_1,V_2$ in $\mathcal{A}$ such that $\widehat{\tau}(V_1),\widehat{\tau}(V_2)$ are linearly independent. Let $W_i=\widehat{\tau}(V_i)$ for $i=1,2$. Then $W_1$ and $W_2$ are positive, linearly independent with $\widehat{\tau}(W_i)=W_i$ for $i=1,2$. By Lemma \ref{outer:lem6}, there exists a non-zero positive $W_3\in\text{span}\{W_1,W_2\}$ which is singular. But $W_3=\widehat{\tau}(W_3)$ has to be strictly positive, leading to a contradiction. Therefore, $\widehat{\tau}$ has rank $1$, validating part (2).
	
	Furthermore, the range of $\widehat{\tau}$ is spanned by the strictly positive element $L=\widehat{\tau}({1})=\sum B_j^* B_j$. Therefore,
	\begin{equation}\label{outer:eq13}
		\forall X\in\mathcal{A}, \exists\,\psi(X)\in\mathbb{C}\text{ such that }\widehat{\tau}(X)=\psi(X)\, L.
	\end{equation}
	Since $\widehat{\tau}$ is linear and positive, $\psi$ is a  positive linear functional. As $\widehat{\tau}$ is faithful, so is $\psi$. Furthermore, as $L=\widehat{\tau}({1})=\psi(1)L$, we have $\psi({1})=1$. Therefore, $\psi$ is unital and hence a state. Again, recall from Theorem \ref{outer:thm6} that $\widehat{\tau}\tau=r(\tau)\,\widehat{\tau}$. We infer that $\tau(L)=\tau\widehat{\tau}({1})=r(\tau)\,\widehat{\tau}({1})=r(\tau)\,L$. Since $\widehat{\tau}^2=\widehat{\tau}$, using \eqref{outer:eq13}, we have
	\[\widehat{\tau}^2(X) =\widehat{\tau}(\widehat{\tau}(X))=\widehat{\tau}(\psi(X)L)=\psi(X)\widehat{\tau}(L)=\psi(X)\psi(L)\,L=\widehat{\tau}(X)\psi(L).\]
	Thus, we must have $\psi(L)=1$. Finally, it is evident that for all $X\in\mathcal{A}$,
	$$\widehat{\tau}\tau(X)=r(\tau)\,\widehat{\tau}(X)\implies \psi(\tau(X))\,L=r(\tau)\,\psi(X)\,L\implies \psi(\tau(x))=r(\tau)\,\psi(X).$$ Altogether, it establishes part (4).
	
	Recall from remark \ref{outer:note1} that if $X\in\mathcal{A}$ is a generalized eigenvector of $\tau$ corresponding to an eigenvalue of modulus less $r(\tau)$, then $\widehat{\tau}(X)=0$. Since $\widehat{\tau}$ is faithful here, there is no such element which is positive. This proves part (5).
\end{proof}

We also note that a faithful state as in this theorem means that
there exists a strictly positive density matrix $R$ such that
$\psi(X)= \text{trace}(RX)$ for all $X\in \mathcal{A}$.

\begin{cor}\label{outer:cor5}
	Let $\tau$ be a CP map on a finite dimensional $C^*$-algebra
	$\mathcal{A}$ with $r(\tau)>0$. If $\tau$ is irreducible, then there
	is a CP map $\sigma$ on $\mathcal{A}$ which is elementarily similar
	to $\tau$ such that $\lVert\sigma\rVert=r(\tau)$.
\end{cor}
\begin{proof}
	From Theorem \ref{outer:thm9}, we have strictly positive $L\in\mathcal{A}$ such that $\tau(L)=r\,L$. Due to Theorem \ref{outer:thm4}, there exists a CP map $\sigma$ on $\mathcal{A}$ which is elementarily similar to $\tau$ such that $\lVert\sigma\rVert=r(\tau)$.
\end{proof}

The converse of Corollary \ref{outer:cor5} does not hold true in
general, as shown by the following counter example.

\begin{eg}
	Let $\mathcal{A}=M_{n_1}\oplus\cdots\oplus M_{n_d}$ be a finite dimensional $C^*$-algebra and choose $A\in\mathcal{A}$ such that $\left\|\frac{A^n}{r^n}\right\|$ forms a bounded sequence where $r=r(A)$. Let us consider the CP map $\tau=\alpha_A$ on $\mathcal{A}$. If $\mathcal{A}$ is a proper sub-algebra of $M_m$ ($m=n_1+\cdots+n_d$), $A$ will not generate the full matrix algebra $M_m$. First, we will prove that there is $\sigma$ which is elementarily similar to $\tau$ with $\|\sigma\|=r(\tau)$. Then, we will show that it does not satisfy the  hypothesis of Corollary \ref{outer:cor5}.
	
	Due to Theorem \ref{outer:thm15}, there exists positive, invertible $V\in\mathcal{A}$ such that $\| V^{-1}AV\|=r$. This implies
	\begin{equation}\label{outer:eq15}
		r^2\,I-(V^{-1}AV)^*(V^{-1}A V)\ge0\implies VA^*V^{-2}AV\le r^2\,I\implies A^*V^{-2} A\le r^2\, V^{-2}.
	\end{equation}
	Now, let $L=V^{-2}$ which is a positive and invertible element of $\mathcal{A}$. Recall that for elementary CP map $\tau=\alpha_A$, we have $r(\tau)=r(A)^2=r^2$. Therefore, \eqref{outer:eq15} shows that $\tau(L)\le r(\tau)\, L$. Using Theorem \ref{outer:thm4}, we get that there exists a CP map $\sigma$ which is elementarily similar to $\tau$ such that $\|\sigma\|=r(\tau)$.
	
	Now, since $A\in M_{n_1}\oplus\cdots\oplus M_{n_d}$, the canonical extension has the same Choi-Kraus decomposition $\widetilde{\tau}(X)=A^*XA$, $X\in M_m$. As $A$ does not generate the algebra $M_m$, it follows from Proposition \ref{outer:irreducible} that $\widetilde{\tau}$ is reducible; hence, so is $\tau$ by Proposition \ref{outer:prop8}.
\end{eg}

For the irreducible CP maps in Examples \ref{outer:eg2} and \ref{outer:eg4}, we now proceed to compute $\widehat{\tau}$, $\psi$, and the strictly positive eigenvector as outlined in Theorem \ref{outer:thm9}. Additionally, we determine the elementarily similar CP map $\sigma$ as described in Corollary \ref{outer:cor5}.

\begin{eg}
	Let us recall the irreducible CP map $\tau$ on $\mathcal{A}=M_2\oplus M_1$ from Example \ref{outer:eg2}, which is defined as follows:
	\[\tau\left(\begin{bmatrix}
		a & b & 0\\
		c & d & 0\\
		0 & 0 & e
	\end{bmatrix}\right)=\begin{bmatrix}
		a+e & 0 & 0\\
		0 & a+e & 0\\
		0 & 0 & d
	\end{bmatrix}.\]
	Since $\tau$ is irreducible, it is noted in Theorem \ref{outer:thm9}
	that $d_{\tau}=1$, implying $\|\tau^n\|_\beta=r^n$ for all but finitely many $n$ where $r$ is the
	spectral radius of $\tau .$ Thus,
	$\widehat{\tau}=\lim\limits_{N\to\infty}\frac{1}{N}\sum\limits_{n=1}^N\frac{\tau^n}{r(\tau)^n}$.
	By induction, we can prove that for all $n\ge4$,
	\[\tau^n\left(\begin{bmatrix}
		a & b & 0\\
		c & d & 0\\
		0 & 0 & e
	\end{bmatrix}\right)=\begin{bmatrix}
		x_{n-2}\,(a+e)+x_{n-3}\,d & 0 & 0\\
		0 & x_{n-2}\,(a+e)+x_{n-3}\,d & 0\\
		0 & 0 & x_{n-3}\,(a+e)+x_{n-4}\,d
	\end{bmatrix}.\]
	where $x_n$ denotes the $n^\text{th}$ term of the Fibonacci sequence
	with $x_0=1,x_1=2$. Recall that,
	$x_n=\frac{p^{n+2}-q^{n+2}}{\sqrt{5}}$ where
	$p=\frac{1+\sqrt{5}}{2}$ and $q=\frac{1-\sqrt{5}}{2}$. Then,
	$r=\lim\limits_{n\to\infty}\lVert{\tau}^n(I)\rVert^{1/n}=\lim\limits_{n\to\infty}x_n^{1/n}=p$.
	Since $\lim\limits_{n\to\infty}\frac{x_n}{r^n}=\frac{r^2}{\sqrt{5}}$
	and $r^2-r-1=0$ (which implies in particular $r^{-1}=r-1$,
	$r^{-2}=2-r$), we infer that $\frac{\tau^n}{r^n}$ converges as $n$
	tends to infinity and hence $\widehat{\tau}=\lim _{n\to
		\infty}\frac{\tau ^n}{r^n}.$ Then, direct computation yields
	\[\widehat{\tau}\left(\begin{bmatrix}
		a & b & 0\\
		c & d & 0\\
		0 & 0 & e
	\end{bmatrix}\right)=
	\frac{1}{\sqrt{5}}(a+(r-1)d+e)
	\begin{bmatrix}
		1 & 0 & 0\\
		0 & 1 & 0\\
		0 & 0 & (r-1)
	\end{bmatrix}.\]
	In particular, $\widehat{\tau}$ has the form $X\mapsto \psi (X)L$ for a
	faithful state $\psi $ and
	\[L=\widehat{\tau}(I)=\frac{1+r}{\sqrt{5}}\begin{bmatrix}
		1 & 0 & 0\\
		0 & 1 & 0\\
		0 & 0 & r-1
	\end{bmatrix}=\frac{1}{\sqrt{5}}\begin{bmatrix}
		r^2 & 0 & 0\\
		0 & r^2 & 0\\
		0 & 0 & r
	\end{bmatrix},\]
	which is positive, invertible and belongs to the eigenspace  of
	$\tau$ corresponding to the eigenvalue $r$. Let
	$\sigma=\alpha_V^{-1}\circ\tau\circ\alpha_V$ where $V=L^{1/2}$. Then,
	\[\sigma\left(\begin{bmatrix}
		a & b & 0\\
		c & d & 0\\
		0 & 0 & e
	\end{bmatrix}\right)=\begin{bmatrix}
		a+(r-1)\,e & 0 & 0\\
		0 & a+(r-1)\,e  & 0\\
		0 & 0 & r\,d
	\end{bmatrix}.\]
	We can easily verify that $\|\sigma\|=\|\sigma(I)\|=r=r(\tau)$.
\end{eg}

\begin{eg}
	Let us revisit the irreducible CP map $\tau$ on
	$\mathcal{A}=M_1\oplus M_1\oplus M_1$, considered in Example
	\ref{outer:eg4}. It is defined as follows:
	\[\tau\left(\begin{bmatrix}
		a & 0 & 0\\
		0 & b & 0\\
		0 & 0 & c
	\end{bmatrix}\right)=\begin{bmatrix}
		b & 0 & 0\\
		0 & a+c & 0\\
		0 & 0 & b
	\end{bmatrix}.\]
	
	It is easy to see that $\tau $ is diagonalizable with eigenvalues
	$\{0, \pm\sqrt{2}\}.$ Therefore $\tau $ has spectral radius
	$\sqrt{2}$ and  $\|\tau ^n\|_{\beta }= 2^{\frac{n}{2}}$ for all but finitely many $n$. Now,
	$\widehat{\tau }$ can be computed using its definition. It maps into the
	eigenspace of $\tau $ with eigenvalue $\sqrt{2}$ and is given by
	\[\widehat{\tau}\left(\begin{bmatrix}
		a & 0 & 0\\
		0 & b & 0\\
		0 & 0 & c
	\end{bmatrix}\right)=\frac{1}{2}(a+b\sqrt{2}+c)\begin{bmatrix}
		1 & 0 & 0\\
		0 & \sqrt{2} & 0\\
		0 & 0 &1
	\end{bmatrix}.\]
	Once again $\widehat{\tau}$ is of the form described in Theorem
	\ref{outer:thm9}.
\end{eg}

\section*{Acknowledgments}
Bhat gratefully acknowledges funding from  SERB (India) through J C
Bose Fellowship No. JBR/2021/000024. Sahasrabudhe thanks the Indian
Statistical Institute for hosting her through  Visiting Scientist
Fellowship, through J C Bose Fellowship of Bhat mentioned
above and DST-WISE fellowship.

\end{document}